\documentclass[10pt]{amsart}

\usepackage{amssymb}
\usepackage{amsbsy}
\usepackage{amscd}
\usepackage[mathscr]{eucal}
\usepackage{verbatim}
\usepackage{version}
\usepackage{color}
\usepackage[matrix,arrow]{xy}
\usepackage{graphicx}

\def\cal{\mathcal}
\def\Bbb{\mathbb}

\newenvironment{NB}{
\color{red}{\bf NB}. \footnotesize 
}{}
\excludeversion{NB}
\newenvironment{NB2}{
\color{blue}{\bf NB}. \footnotesize
}{}

\newcommand{\GL}  {\operatorname{GL}}

\newcommand{\ch}{\operatorname{ch}}

\newcommand{\Coh}{\operatorname{Coh}}

\newcommand{\Ext}{\operatorname{Ext}}
\newcommand{\Hom}{\operatorname{Hom}}

\newcommand{\rk}{\operatorname{rk}}

\newcommand{\NS}{\operatorname{NS}}
\newcommand{\coker}{\operatorname{coker}}

\newcommand{\Amp}{\operatorname{Amp}}

\newcommand{\alg}{\operatorname{alg}}

\newcommand{\Stab}{\operatorname{Stab}}

\font\b=cmr10 scaled \magstep5
\def\bigzerou{\smash{\lower1.7ex\hbox{\b 0}}}

\numberwithin{equation}{section}
\theoremstyle{plain}
 \newtheorem{thm}{Theorem}[section]
 \newtheorem{lem}[thm]{Lemma}
 \newtheorem{prop}[thm]{Proposition}
 \newtheorem{cor}[thm]{Corollary}

\theoremstyle{definition}
 \newtheorem{defn}[thm]{Definition}
 
\theoremstyle{remark}
 \newtheorem{rem}[thm]{Remark}

\begin{document}

\title{Preservation of stability under
the Fourier-Mukai transform whose kernel is the Poincar\'{e} line bundle}
\author{K\={o}ta Yoshioka}
\address{Department of Mathematics, Faculty of Science,
Kobe University,
Kobe, 657, Japan
}
\email{yoshioka@math.kobe-u.ac.jp}

\thanks{
The author is supported by the Grant-in-aid for 
Scientific Research (No. 21H04429, 
23K03053, 26K06742), JSPS}

\begin{abstract}
For a Fourier-Mukai transform whose kernel is the Poincar\'{e} line bundle,
we study the preservation of Gieseker stability of sheaves on any abelian surface. 
As an application, we give remarks on the weak Brill-Noether property.
\end{abstract}

\maketitle

\renewcommand{\thefootnote}{\fnsymbol{footnote}}

\section{Introduction}

Let $X$ be an abelian surface over ${\Bbb C}$ and
let $(H^{2*}(X,{\Bbb Z}),\langle\;\;,\;\; \rangle)$ be the Mukai lattice of $X$:
Thus 
$H^{2*}(X,{\Bbb Z}):={\Bbb Z} \oplus H^2(X,{\Bbb Z}) \oplus {\Bbb Z}$ and
\begin{equation*}
\langle x,y \rangle:=(x_1 \cdot y_1)-x_0 y_2-y_0 x_2 \in {\Bbb Z}
\end{equation*}
for $x=(x_0,x_1,x_2), y=(y_0,y_1,y_2) \in  {\Bbb Z} \oplus H^2(X,{\Bbb Z}) \oplus {\Bbb Z}$.
Mukai lattice has  a natural Hodge structure and
$H^{2*}(X,{\Bbb Z})_{\alg}:={\Bbb Z} \oplus \NS(X) \oplus {\Bbb Z}$ is the algebraic
part of $H^{2*}(X,{\Bbb Z})$.
For an object $E$ of the bounded derived category ${\bf D}(X)$ of
coherent sheaves, 
$v(E):=\ch(E) \in H^{2*}(X,{\Bbb Z})$ is the Mukai vector of $E$. 
We also say that $v \in H^{2*}(X,{\Bbb Z})_{\alg}$ is a Mukai vector.
For an ample divisor $H$ and a Mukai vector $v$,
${\cal M}_H(v)$ denotes the moduli stack of semi-stable sheaves $E$ with $v(E)=v$.
There are many results on semi-stable sheaves and also their moduli stacks ${\cal M}_H(v)$.
For examples, ${\cal M}_H(v) \ne \emptyset$ if and only if $\langle v^2 \rangle \geq 0$
and ${\cal M}_H(v)$ is an irreducible normal stack of dimension $\langle v^2 \rangle+1$, where
$H$ is general with respect to $v$.   
For the proof of these results, we used the symmetry of the derived category of coherent sheaves, that is,
the Fourier-Mukai transforms.
Since they are the transforms of the derived categories of coherent sheaves,
Gieseker stability is not preserved in general.
So we need to
find some conditions for the preservation of stability, or to replace 
the kernel of  the Fourier-Mukai transform for the preservation of stability.

For the Fourier-Mukai transform
$\Phi_{X \to \widehat{X}}^{{\cal P}^{\vee}}:{\bf D}(X) \to {\bf D}(\widehat{X})$,
we proved in \cite{Y:Stability}
that the stability is preserved if the Picard number is 1, where
$\widehat{X}$ is the dual abelian surface and
${\cal P}$ is the Poincar\'{e} line bundle on $X \times \widehat{X}$.
In this paper, we shall treat the case where the Picard number is not 1.
The following generalization of \cite[Thm. 3.14]{Y:Stability} is our main result.
\begin{NB}
In \cite{Y:abel},
we constructed a birational map of moduli space by a modification of the kernel.
For the modification, we used a characterization of the Albanese map in \cite{Y:7}
to study the Fourier-Mukai transform of a general member $E$ of ${\cal M}_H(v)$.
In particular we constructed a birational map of moduli stacks.
In \cite{MYY:2018}, \cite{YY2}, \cite{Y}, we explained the birationality of
moduli spaces in terms of wall crossing of Bridgeland stability conditions.
In particular if the Picard number is 1, then the birational map in \cite{Y:abel} is the same 
as the birational map induced by totally semistable walls.

For the preservation of Gieseker stability, we studied extensively if the Picard number is 1.
We proved that the stability is preserved for a general member of
moduli spaces.
On the other hand, if the Picard number is not 1,
not much is known.

On the other hand if ${\cal P}$ is a family of line bundles and the Picard number $\rho(X)=1$,
then the Fourier-Mukai transform preserves the stability for a general member of the moduli stack.

\end{NB}

\begin{thm}[{Theorem \ref{thm:stability}}]\label{thm:main}
Let $v=(r,\xi,a)$ be a Mukai vector such that $r > 0$ and $(\xi \cdot H)>0$, where 
$H$ is an ample divisor on $X$.
We set $\ell:=\langle v^2 \rangle/2$.
\begin{enumerate}
\item[(1)]
Assume that $a>0$. 
If $X$ is not a product of elliptic curves or $\xi$ is not primitive,
then there are ample divisors $L \in \NS(X)$ and $L' \in \NS(\widehat{X})$ such that 
$\Phi_{X \to \widehat{X}}^{{\cal P}^{\vee}}(E)^{\vee}$ is stable with respect to $L'$ for a general
$E \in {\cal M}_L(v)$.
\item[(2)]
Assume that $a \leq 0$ and
$v \ne (\ell,kC,-1), (1,kC,-\ell)$, where 
$C$ is an elliptic curve and $k \geq \ell+1$. 
Then there are ample divisors $L \in \NS(X)$ and $L' \in \NS(\widehat{X})$ such that 
$\Phi_{X \to \widehat{X}}^{{\cal P}^{\vee}}(E)[1]$ is stable with respect to $L'$ for a general
$E \in {\cal M}_L(v)$.
\end{enumerate}
\end{thm}

We would like to remark that the preservation of stability is related to the weak Brill-Noether property.
Indeed if $\Phi_{X \to \widehat{X}}^{{\cal P}^{\vee}}(E)^{\vee} \in \Coh(\widehat{X})$ (resp.
$\Phi_{X \to \widehat{X}}^{{\cal P}^{\vee}}(E)[1] \in \Coh(\widehat{X})$), then
there is a point $\hat{x} \in \widehat{X}$ such that
$H^i(X,E \otimes {\cal P}_{|X \times \{\hat{x} \}}^{\vee})=0$ for $i \ne 0$
(resp. $H^i(X,E \otimes {\cal P}_{|X \times \{\hat{x} \}}^{\vee})=0$ for $i \ne 1$).
In \cite{CNY2}, Coskun, Nuer and the author proved 
that the weak Brill-Noether property 
does not hold in general.
In particular there is a Mukai vector $v$ and an ample divisor $L \in \NS(X)$
such that the cohomology sheaves $H^0(\Phi_{X \to \widehat{X}}^{{\cal P}^{\vee}}(E))$ and
$H^1(\Phi_{X \to \widehat{X}}^{{\cal P}^{\vee}}(E))$ are of positive rank  
for all $E \in {\cal M}_L(v)$.
On the other hand, our main result almost says that if we choose the polarization suitably,
we have the weak Brill-Noether property.
Indeed in the course of the proof of Theorem \ref{thm:main},
we get an affirmative result on the weak Brill-Noether property.
\begin{prop}[{Proposition \ref{prop:wBN-L}}]\label{prop:main}
Let $v=(r,\xi,a)$ be a Mukai vector such that $r>0$, $(\xi \cdot H)>0$ and $\langle v^2 \rangle \geq 0$.
Then there is an ample divisor $L$ (depending on $v$) such that the weak Brill-Noether property holds for
${\cal M}_L(v)$.  
Thus there is $E \in {\cal M}_L(v)$ such that $E$ has at most one nonzero cohomology group.
\end{prop}

In order to explain our counter example for the weak Brill-Noether property in \cite{CNY2},
let ${\cal M}(v)$ be the moduli stack of coherent sheaves $E$
such that $v(E)=v$ and $E$ is semistable with respect to a general ample divisor.
Thus ${\cal M}(v)=\cup_L {\cal M}_L(v)$, where $L$ runs the set of general polarizations with respect to $v$
(see Definition \ref{defn:stable}).
Then the counter examples exists if ${\cal M}(v)$ is not irreducible.
In this case, if the boundary $\partial \! \Amp(X)$ of the ample cone contains an irrational ray,
we have infinitely many irreducible components ${\cal M}_L(v)$
and the weak Brill-Noether property does not hold if $L$ is close to $\partial \! \Amp(X)$.
In this paper, we shall give positive results on the preservation of stability
and the weak Brill-Noether property for a fixed ample divisor.

(I) Under the irreducibility of moduli stacks, we do not need to care about polarizations.
So we get following result from Theorem \ref{thm:main}.
\begin{cor}[{Corollary \ref{cor:stability}}]\label{cor:main}
Let $v=(r,\xi,a)$ be a primitive Mukai vector such that $r>0$ and $(\xi \cdot H)>0$.
Let $\widehat{H}$ be the ample divisor on $\widehat{X}$ which is naturally associated to $H$, that is,
the Poincar\'{e} dual of $H$. 
\begin{enumerate}
\item[(1)]
Assume that $a>0$ and $\langle v^2 \rangle \geq 2r, 2a$.
If $X$ is not a product of elliptic curves or $\xi$ is not primitive, then
$\Phi_{X \to \widehat{X}}^{{\cal P}^{\vee}}(E)^{\vee}$ is stable with respect to $\widehat{H}$ for a general
$E \in {\cal M}_H(v)$.
\item[(2)]
Assume that $(\xi^2)>0$ and $a< 0$.
Then $\Phi_{X \to \widehat{X}}^{{\cal P}^{\vee}}(E)[1]$ is stable with respect to $\widehat{H}$ for a general
$E \in {\cal M}_H(v)$.
\end{enumerate}
\end{cor}

(II)
It is natural to study the weak Brill-Noether property when the ample divisor is not close to 
$\partial \! \Amp(X)$.
For a Mukai vector $v=(r,\xi,a)$ with $\xi \in \Amp(X)$,
a natural choice of the ample divisor is $\xi$.
In this case, we get the following nice result.

\begin{thm}[{Theorem \ref{thm:wBN}}]\label{thm:main2}
Assume that $v=(r,dH,a)$.
\begin{enumerate}
\item[(1)] 
If $d \geq 0$,
the weak Brill-Noether property holds.
Thus there is $E \in {\cal M}_H(v)$ such that $E$ has at most one nonzero cohomology
group.
\item[(2)]
If $d<0$, then
the weak Brill-Noether property holds unless $v=(r,0,-1)e^{kH}$, where
$k$ is a negative integer.
\end{enumerate}
\end{thm}

For the proof of these results, we use Bridgeland stability conditions as in \cite{CNY2}.
Thus we study the wall crossing behavior for totally semistable walls in the space $\Stab(X)$ of Bridgeland
stability conditions.
Under the assumption of Theorem \ref{thm:main}, we shall find a stability condition such that
a general stable object $E$ is torsion free and its Fourier-Mukai transform is also torsion free.
If a general stable object $E$ is torsion free, then we can find an ample divisor
such that $E$ is Gieseker stable, and hence our main result follows.

Let us explain the organization of this paper.
In section \ref{sect:pre},
we explain several notation and basic results.
In subsection \ref{subsect:chamber}, we introduce totally semi-stable walls and
study irreducible components parameterizing stable sheaves.
We also recall a result on the irreducible components in \cite{CNY2}.
In section \ref{sect:Bridgeland},
we recall Bridgeland stability conditions on an abelian surface. 
In particular, we explain totally semi-stable walls and
the relation of adjacent chambers.
In section \ref{sect:Preserve}, we prove our result.
In subsection \ref{subsect:line},
we study wall crossing behavior along a line connecting two chambers,
one is the chamber corresponding to the Gieseker semi-stability, and 
the other is the chamber whose transform by 
$\Phi_{X \to \widehat{X}}^{{\cal P}^{\vee}}$ is related to Gieseker semi-stability. 
In subsection \ref{subsect:proof}, we prove Theorem \ref{thm:main} and
Proposition \ref{prop:main}.
In subsection \ref{subsect:dH}, we treat the case where $v=(r,dH,a)$. In particular,
we prove Theorem \ref{thm:main2}.

\begin{NB}
In Proposition \ref{prop:t_0}, 
we divide the line into at most three parts by the property of semi-stable objects: 
\begin{enumerate}
\item
a general semi-stable object is a torsion free coherent sheaf,
\item
a general semi-stable object is
a coherent sheaf with a torsion subsheaf, and 
\item
a general semi-stable object is
a two term complex of coherent sheaves.    
\end{enumerate}
\end{NB}

\section{Preliminaries.}\label{sect:pre}

In this paper, $H$ denotes an ample divisor on an abelian surface $X$.  
Let $\Amp(X)$ be the ample cone of $X$.
Then
$$
\Amp(X)=\{L \in \NS(X) \mid (L^2)>0, (L \cdot H)>0 \}.
$$
We note that
$D \in \NS(X)$ is effective, that is,
$D$ is represented by an effective divisor iff
\begin{enumerate}
\item
$(D^2)>0$ and $(D \cdot H)>0$ or
\item
$(D^2)=0$ and $(D \cdot H)>0$.
\end{enumerate}

\begin{rem}
For case (i), every divisor $C$ representing $D$ is effective. 
\end{rem}

\subsection{Fourier-Mukai transforms}
Let $\widehat{X}$ be the dual abelian surface of $X$ and
${\cal P}$ the Poincar\'{e} line bundle on $X \times \widehat{X}$.
Let $\Phi_{X \to \widehat{X}}^{{\cal P}^{\vee}}:{\bf D}(X) \to {\bf D}(\widehat{X})$
be the Fourier-Mukai transform whose kernel is ${\cal P}^{\vee}$ \cite{Mukai:1981}:
\begin{equation}
\begin{matrix}
\Phi_{X \to \widehat{X}}^{{\cal P}^{\vee}}: & {\bf D}(X) & \to & {\bf D}(\widehat{X})\\
& E & \mapsto & {\bf R}p_{\widehat{X}*}({\cal P}^{\vee} \otimes p_X^*(E)),
\end{matrix}
\end{equation}
where $p_X:X \times \widehat{X} \to X$ and
$p_{\widehat{X}}:X \times \widehat{X} \to \widehat{X}$
are projections.
We set
$\Phi:=\Phi_{X \to \widehat{X}}^{{\cal P}^{\vee}}$ and
$\widehat{\Phi}:=\Phi_{\widehat{X} \to X}^{{\cal P}}$.
We also set 
\begin{equation}
\begin{split}
\Phi^i(E):=& H^i(\Phi(E)) \in \Coh(\widehat{X}),\; E \in \Coh(X)\\
\widehat{\Phi}^i(F):=& H^i(\widehat{\Phi}(F)) \in \Coh(X),\;F \in \Coh(\widehat{X}).
\end{split}
\end{equation}

$\Phi$ induces an isometry of Hodge structures
\begin{equation}
\begin{matrix}
H^{2*}(X,{\Bbb Z}) & \to & H^{2*}(\widehat{X},{\Bbb Z})\\
(r,\xi,a) & \mapsto & (a,-\widehat{\xi},r),
\end{matrix}
\end{equation}
where $\widehat{\xi}$ is the Poincar\'{e} dual of $\xi$.
If $\Phi^i(E)$ is a torsion sheaf, then
$h^i(E \otimes {\cal P}_{|X \times \{ y \}}^{\vee})=0$
for a general $y \in \widehat{X}$.
Hence the Fourier-Mukai transform can be used to study
cohomology groups of a general stable sheaf.

The following follows from the structure of Fourier-Mukai transforms
on an abelian surface \cite{BM}, \cite{Or1}. See also \cite[Prop. 4.4]{Y:abel}.
\begin{lem}[{cf. \cite[Lem. 2.9]{CNY2}}]\label{lem:semihom}
Let $E$ be a semi-stable sheaf with an isotropic Mukai vector $v=(r,\xi,a)$.
\begin{enumerate}
\item[(1)]
Assume that $(\xi^2)>0$. 
\begin{enumerate}
\item
If $(\xi \cdot H)>0$, then $\Phi(E) \in \Coh(\widehat{X})$. 
\item
If $(\xi \cdot H)<0$, then $\Phi(E)[2] \in \Coh(\widehat{X})$.
\end{enumerate}
\item[(2)]
Assume that $(\xi^2)<0$. Then $\Phi(E)[1] \in \Coh(\widehat{X})$. 
\item[(3)]
Assume that $(\xi^2)=0$ with $\xi \ne 0$.
\begin{enumerate}
\item
If $r=0$, then $\xi$ is effective.
Hence $\Phi(E) \in \Coh(\widehat{X})$ for $a> 0$
and $\Phi(E)[1] \in \Coh(\widehat{X})$ for $a \leq 0$.
\item 
If $r>0$, then $a=0$. If $(\xi \cdot H)>0$, then
$\Phi(E)[1] \in \Coh(\widehat{X})$. 
If $(\xi \cdot H)<0$, then
$\Phi(E)[2] \in \Coh(\widehat{X})$. 
\end{enumerate} 
\end{enumerate}
\begin{NB}
\begin{enumerate}
\item
If $\Phi(F) \in \Coh(\widehat{X})$, then
$(\xi^2)>0$ and
$\xi$ is effective. 
\item
If $\Phi(F)[2] \in \Coh(\widehat{X})$, then
$-\xi$ is effective. 
 \end{enumerate} 
\end{NB}
\end{lem}

\begin{rem}\label{rem:semi-hom}
$E$ is a semi-stable sheaf with an isotropic Mukai vector if and only if
$E$ is a semi-homogeneous sheaf.
\end{rem}

\begin{NB}
\begin{proof}
We note that there is a number $k \in \{ 0,1,2 \}$ such that $\Phi(F)[k]$
is a semi-homogeneous sheaf (\cite[Cor. 2.10]{BM} or \cite{Or1}). 
If $(\xi^2)=2ra \ne 0$, then 
$\rk \Phi(F)=a \geq 0$ implies $(\xi^2) \geq 0$.
Assume that $(\xi^2)>0$.
Then $\xi$ is ample or $-\xi$ is ample.
Moreover $\Phi^2(F)=0$ if $\xi$ is ample and
$\Phi^0(F)=0$ if $-\xi$ is ample. 

We next assume that $(\xi^2)=0$ and $\xi \ne 0$.
Then $\xi$ is nef or $-\xi$ is nef and $ra=0$.
Assume that $r=0$. Then $\xi$ is effective.
Hence $\Phi(E) \in \Coh(\widehat{X})$ for $a>0$ and
$\Phi(E)[1] \in \Coh(\widehat{X})$ for $a \leq 0$.
Assume that $a=0$. 
If $\xi$ is nef, then $\Phi^2(F)=0$.
If $-\xi$ is nef, then $h^0(F \otimes {\cal P}_t^{\vee})=0$
for all $t \in \widehat{X}$.
Therefore our claim holds.
\begin{NB2}
There is a covering $\pi:Y \to X$ and a line bundle $L$ on $Y$
such that $\pi_*(L)=F$.
If $(c_1(L)^2) >0$, then
$L$ is ample or $L^{\vee}$ is ample.
If $L$ is ample, then $0 \ne H^0(Y,L)=H^0(X,F)$.
$L^{\vee}$ is ample, then $0 \ne H^0(Y,L)=H^2(X,F)$.
If $(c_1(L)^2)<0$, then
$0 \ne H^1(Y,L)=H^1(X,F)$.
Assume that $(c_1(L)^2)=0$.
Then $c_1(L)$ is nef or $-c_1(L)$ is nef.
If $c_1(L)$ is nef, then
there is $t \in \widehat{X}$ such that
$0 \ne H^1(Y,L \otimes \pi^*({\cal P}_t))=H^1(X,F \otimes {\cal P}_t)$.
Hence $\Phi(E) \not \in \Coh(\widehat{X})$.
If $-c_1(L)$ is nef, then
there is $t \in \widehat{X}$ such that
$0 \ne H^2(Y,L \otimes \pi^*({\cal P}_t))=H^2(X,F \otimes {\cal P}_t)$.
\end{NB2}
\end{proof}
\end{NB}

\begin{NB}
For a divisor $\xi$ with $(\xi^2)=0$,
$\Phi({\cal O}_X(\xi)) \ne 0$ implies there is $t \in \widehat{X}$ such that 
$h^0({\cal O}_X(\xi) \otimes {\cal P}_t)>0$ or
$h^2({\cal O}_X(\xi) \otimes {\cal P}_t)>0$.
Hence $\xi$ is nef or $-\xi$ is nef.
\end{NB}


\subsection{The dependence of Gieseker semistability on polarizations}\label{subsect:chamber}

For the $\mu$-semi-stability of coherent sheaves, we have the notion of
wall and chamber in $\Amp(X)_{\Bbb R}$.
An ample divisor $H \in \Amp(X)$ is general, if $H$ is in a chamber, which is equivalent to the following
definition. 

\begin{defn}
An ample divisor $H$ is general with respect to $v$
if 
$$
\frac{(c_1(E_1) \cdot H)}{\rk E_1}=\frac{(c_1(E) \cdot H)}{\rk E}
\iff \frac{v(E_1)}{\rk E_1}=\frac{v(E)}{\rk E}
$$ 
for any subsheaf $E_1$ of a $\mu$-semi-stable sheaf $E$ with $v(E)=v$.
\end{defn}

\begin{defn}\label{defn:stable}
We set 
$$
\Amp(X)_v:=\{ H \in \Amp(X) \mid \text{$H$ is general with respect to $v$} \}.
$$
Let
$$
{\cal M}(v):=\cup_{H \in \Amp(X)_v } {\cal M}_H(v)
$$
be the moduli stack of semi-stable sheaves $E$ with $v(E)=v$.
\end{defn}
${\cal M}(v)$ is a normal stack of dimension $\langle v^2 \rangle+1$.
Moreover ${\cal M}(v)$ is irreducible if $v$ is not primitive.
Assume that $v$ is primitive.
Then ${\cal M}(v)$ is smooth, but it is not irreducible in general.

\begin{NB}
For a general ample divisor $L$ with respect to $v$,
${\cal M}_L(v)$ consists of stable sheaves.
By the openness of stability, ${\cal M}_L(v)$ is an open substack of ${\cal S}(v)$.
Hence the closure of 
${\cal M}_L(v)$ in ${\cal S}(v)$ is an irreducible component. 
\end{NB}

\begin{prop}\label{prop:chamber}
Let $H_1$ and $H_2$ be ample divisors which are general with respect to $v=(r,\xi,a)$.
\begin{enumerate}
\item[(1)]
Assume that ${\cal M}_{H_1}(v) \cap {\cal M}_{H_2}(v)=\emptyset$.
Then there are isotropic vector $v_1=(r_1,\xi_1,a_1)$, $v_2=(r_2,\xi_2,a_2)$ and a decomposition
\begin{equation}\label{eq:chamber}
v=\ell v_1+v_2, \; \langle v_1,v_2 \rangle=1,\; r_1,r_2>0,\;
 ((r_2 \xi_1-r_1 \xi_2) \cdot H)=0, 
\end{equation}
where $H$ is an ample divisor. 
\item[(2)]
If $\langle v^2 \rangle \geq 2r$, then 
${\cal M}_{H_1}(v) \cap {\cal M}_{H_2}(v) \ne \emptyset$.
In particular ${\cal M}(v)$ is irreducible.
\end{enumerate}
\end{prop}

\begin{proof}
(1)
We may assume that $H_1$ and $H_2$ are separated by a wall $W$. 
Let $H$ be a general ample divisor on $W$.
By the argument in \cite[3.3]{KY}, the claim follows. 
For \cite[Lem. 3.9]{KY}, we would like to remark that
the lattice $L$ in \cite[Lem. 3.9]{KY} is the orthogonal complement of $H$
(i.e., $L=H^\perp$) and the condition $\xi_i \in L$ should be replaced by
$r_2 \xi_1-r_1 \xi_2 \in L$.

(2)
Assume that there is a decomposition of $v$ in \eqref{eq:chamber}.
Since $\langle v^2 \rangle=2\ell$ and
$r=\ell r_1+r_2>\ell$,
$2r>\langle v^2 \rangle$.
Hence if $\langle v^2 \rangle \geq 2r$, then 
${\cal M}_{H_1}(v) \cap {\cal M}_{H_2}(v) \ne \emptyset$.
\end{proof}

\begin{defn}
A wall $W (\subset \Amp(X)_{\Bbb R})$ for $v$ is totally semi-stable if
$W=(r_2 \xi_1-r_1 \xi_2)^\perp$ for a decomposition \eqref{eq:chamber}. 
\end{defn}

The following proposition shows that there are infinitely many totally semi-stable walls in general.

\begin{prop}[{ \cite[Prop. 4.9]{CNY2}}]
Assume that $\rho(X) \geq 3$ or $\rho(X)=2$ and $X$ does not contain an elliptic curve.
If there is a decomposition \eqref{eq:chamber}, then
there are infinitely many similar decompositions, and hence there are infinitely many 
${\cal M}_L(v)$ $(L \in \Amp(X)_v)$ defining irreducible components 
in ${\cal M}(v)$.
Moreover there are infinitely many ${\cal M}_L(v)$ satisfying
$\rk H^0(\Phi_{X \to \widehat{X}}^{{\cal P}^{\vee}}(E))>0$ and
$\rk H^1(\Phi_{X \to \widehat{X}}^{{\cal P}^{\vee}}(E))>0$ 
for all $E \in {\cal M}_L(v)$. 
\end{prop}

\begin{proof}
In the proof of \cite[Prop. 4.9]{CNY2}, we only used that there is a decomposition \eqref{eq:chamber}.
So the claim follows from \cite{CNY2}.
\end{proof}

\section{Bridgeland stability.}\label{sect:Bridgeland}

\subsection{Stability conditions on an abelian surface}

Let us briefly recall some notation on stability conditions of Bridgeland
for an abelian surface $X$.
For more details, see \cite{Br:stability}, \cite{Br:3} and \cite{Br:4}.
A stability condition $\sigma=(Z_\sigma,{\cal P}_\sigma)$ 
on ${\bf D}(X)$
consists of a group homomorphism $Z_\sigma: {\bf D}(X) \to {\Bbb C}$
and a slicing ${\cal P}_\sigma$ of ${\bf D}(X)$ such that if
$0 \ne E \in {\cal P}_\sigma(\phi)$ then 
$Z_\sigma(E) =m(E)\exp(\pi \sqrt{-1} \phi)$ 
for some $m(E) \in {\Bbb R}_{>0}$.
The set of stability conditions has a structure of complex 
manifold.
We denote this space by $\Stab(X)$.

\begin{defn}\label{defn:B-stable}
\begin{enumerate}
\item[(1)]
A $\sigma$-semi-stable object $E$ of phase $\phi$ is
an object of  ${\cal P}_\sigma(\phi)$.
If $E$ is a simple object of ${\cal P}_\sigma(\phi)$, then
$E$ is $\sigma$-stable. 
\item[(2)]
For $v \in H^*(X,{\Bbb Z})_{\alg}$,
we denote the moduli stack of $\sigma$-semi-stable objects $E$ with 
$v(E)=v$ by ${\cal M}_\sigma(v)$,
where we usually choose 
$\phi:=\mathrm{Im}(\log Z_\sigma(v))/\pi \in (-1,1]$.
\end{enumerate}
\end{defn}

By \cite[Prop. 5.3]{Br:stability},
giving a stability condition $\sigma$ is
the same as giving a bounded $t$-structure on ${\bf D}(X)$ 
and a stability function $Z_\sigma$ on its heart ${\cal A}_\sigma$
with the Harder-Narasimhan property.
For $\sigma=(Z_\sigma,{\cal P}_\sigma)$,
we have the relation ${\cal A}_\sigma ={\cal P}_\sigma((0,1])$,
where ${\cal P}_\sigma((0,1])$ is the subcategory of ${\bf D}(X)$
generated by semi-stable objects $E \in {\cal P}_\sigma(\phi)$ with 
$\phi \in (0,1]$.  
Since the pair $(Z_\sigma,{\cal A}_\sigma)$
defines a stability condition, 
we also use the symbol $\sigma=(Z_\sigma,{\cal A}_\sigma)$
to denote a stability condition.
For a $\sigma$-semi-stable object $E \in 
{\cal P}_\sigma(\phi)$, we set $\phi_\sigma(E)=\phi$.
For $E \in {\cal P}_\sigma((0,1])$,
we also set $\phi_\sigma(E) \in (0,1]$ 
by $Z_\sigma(E)=m(E)\exp(\pi \sqrt{-1} \phi_\sigma(E))$. 

Let $\Stab(X)$ be the space of stability conditions.
For an equivalence $\Phi:{\bf D}(X) \to {\bf D}(X')$,
we have an isomorphism $\Phi:\Stab(X) \to \Stab(X')$
such that 
$\Phi(\sigma)$ $(\sigma \in \Stab(X))$ 
is a stability condition given by
\begin{equation}\label{eq:FM-action}
\begin{split}
Z_{\Phi(\sigma)}=& Z_\sigma \circ \Phi^{-1}:{\bf D}(X') \to {\Bbb C},\\
{\cal P}_{\Phi(\sigma)}(\phi)=& \Phi({\cal P}_\sigma(\phi)).
\end{split}
\end{equation}
We also have an action of the universal covering 
$\widetilde{\GL}_2^+({\Bbb R})$ of $\GL_2^+({\Bbb R})$ on 
$\Stab(X)$.
Since ${\Bbb C}^{\times} \subset  \GL_2^+({\Bbb R})$,
we have an injective homomorphism
${\Bbb C} \to  \widetilde{\GL}_2^+({\Bbb R})$.
Thus 
we have an action of $\lambda \in {\Bbb C}$ on
$\Stab(X)$. 
For a stability condition $\sigma \in \Stab(X)$, 
$\lambda(\sigma)$ is 
given by
\begin{equation}\label{eq:C-action}
\begin{split}
Z_{\lambda(\sigma)}=&
\exp(-\pi  \sqrt{-1} \lambda)Z_\sigma\\
{\cal P}_{\lambda(\sigma)}(\phi)=&
{\cal P}_\sigma(\phi+\mathrm{Re}\lambda).
\end{split}
\end{equation}

\subsection{Geometric stability conditions}
We set
$$
{\cal H}:=\NS(X)_{\Bbb R} \times \Amp(X)_{\Bbb R}.
$$
For $(\beta,\omega) \in {\cal H}$,
Bridgeland \cite{Br:3} constructed a stability condition
$\sigma_{(\beta,\omega)}=(Z_{(\beta,\omega)},{\cal A}_{(\beta,\omega)})$
which is characterized
by the stability of $k_x$ ($x \in X$), where
$Z_{(\beta,\omega)}(\bullet)=
\langle e^{\beta+\sqrt{-1}\omega},v(\bullet) \rangle:
{\bf D}(X) \to {\Bbb C}$ is the stability function and
${\cal A}_{(\beta,\omega)}$ 
is an abelian category which is a tilting of $\Coh(X)$
by a torsion pair $({\cal T}_{(\beta,\omega)},{\cal F}_{(\beta,\omega)})$:
\begin{enumerate}
\item
${\cal T}_{(\beta,\omega)}$ is a full subcategory of $\Coh(X)$ generated by
torsion sheaves and $\mu$-stable sheaves $E$ with 
$((c_1(E)-\rk E \beta) \cdot \omega)>0$ and
\item
${\cal F}_{(\beta,\omega)}$ is a full subcategory of $\Coh(X)$ generated by
$\mu$-stable sheaves $E$ with 
$((c_1(E)-\rk E \beta) \cdot \omega) \leq 0$.
\end{enumerate}
This kind of stability conditions are called geometric.
To be more precise,
for a geometric stability condition
$\sigma=(Z_\sigma,{\cal A}_\sigma)$,
we require that
$\mathrm{Re}\mho, \mathrm{Im}\mho$ span a positive definite
2-plane of $H^*(X,{\Bbb R})$,
where $Z_\sigma(\bullet)=\langle \mho,\bullet \rangle$
with $\mho \in H^*(X,{\Bbb Q})_{\alg} \otimes {\Bbb C}$. 
Up to the action of $\widetilde{\GL}_2^+({\Bbb R})$, 
there is $(\beta,\omega)$ such that
$\sigma=\sigma_{(\beta,\omega)}$.
We set ${\cal M}_{(\beta,\omega)}(v):=
{\cal M}_{\sigma_{(\beta,\omega)}}(v)$.
\begin{NB}
Old:
$\sigma \in \Stab(X)$ satisfies  
(1)
$Z_\sigma(\bullet)=\langle \mho,\bullet \rangle$
with $\langle \mathrm{Re}\mho^2 \rangle=
\langle \mathrm{Im}\mho^2 \rangle$, 
$\langle \mathrm{Re}\mho,\mathrm{Im}\mho \rangle=0$
and
(2)
$k_x$ ($x \in X$) are $\sigma$-stable
with the phase $\phi_{\sigma}(k_x)=1$. Then
$\sigma=\sigma_{(\beta,\omega)}$ for some $(\beta,\omega)$.
\end{NB}

We collect some results in \cite{YY2} and \cite{Y}.

\begin{defn}\label{defm:moduli}
Let ${\cal M}_{(\beta,\omega)}(v)$ be the moduli stack of
$\sigma_{\beta,\omega}$-stable objects $E$ with $v(E)=v$. 
\end{defn}

\begin{rem}\label{rem:moduli}
${\cal M}_{(\beta,\omega)}(v)$ is well-defined as a scheme, but parameterizing
objects are only determined up shift $[2k]$ $(k \in {\Bbb Z})$.
Indeed if $E$ is $\sigma_{(\beta,\omega)}$-stable, then
$E[k]$ is $\sigma_{(\beta,\omega)}$-stable with $v(E[k])=(-1)^k v(E)$.
In order to fix the shift, we need to fix the phase of $E$.
\end{rem}

For a fixed Mukai vector $v \in H^{2*}(X,{\Bbb Z})$, there exists a locally finite set of \emph{walls} (real codimension one submanifolds with boundary) in ${\cal H}$, 
depending only on $v$, with the following properties (see \cite{Br:3}):
\begin{enumerate}
\item[(1)]
When $\sigma$ varies in a chamber, that is, a connected component of 
the complement of the union of walls, the sets of 
$\sigma$-semistable and $\sigma$-stable objects of class $v$ do not change. 
If $v$ is primitive, then $\sigma$-stability coincides with $\sigma$-semistability 
for $\sigma$ in a chamber for $v$.

\item[(2)]
 When $\sigma$ lies on a wall $W \subset {\cal H}$, 
there is a $\sigma$-semistable object of class $v$ that is unstable in one of the adjacent chambers 
and semistable in the other adjacent chamber.
 If $\sigma= (Z_\sigma, {\cal A}_\sigma)$ lies on a wall, there exists a $\sigma$-semistable object $E$ 
of Mukai vector $v$ and a subobject $F \subset E$ in ${\cal A}_\sigma$ 
with the same $\sigma$-slope but 
$v(F) \not \in {\Bbb R} v$.

\item[(3)]
 Given a polarization $H\in\Amp(X)$ and the Mukai vector $v$ of an $H$-Gieseker semistable sheaf, 
there exists a chamber ${\cal C}$ for $v$, the \emph{Gieseker chamber}, 
where the set of $\sigma$-semistable objects of class 
$v$ coincides with the set of $H$-Gieseker semistable sheaves \cite[Prop. 14.2]{Br:3}.
\end{enumerate}

\begin{defn}\label{defn:W}
For a Mukai vector $v_1$,
let $W_{v_1}$ be a closed subset of ${\cal H}$ such that
$$
W_{v_1}=\{(\beta,\omega) \in {\cal H} \mid {\Bbb R} Z_{\sigma_{(\beta,\omega)}}(v)=
 {\Bbb R}Z_{\sigma_{(\beta,\omega)}}(v_1) \}.
$$
\end{defn}

For a wall $W$ in ${\cal H}$,
there is $v_1$ such that $W= W_{v_1}$.  
We note that a Mukai vector $v_1$ defines a wall in the space of stability conditions
if and only if
\begin{equation}\label{eq:def-wall}
\langle v_1^2 \rangle \geq 0, \langle (v-v_1)^2 \rangle \geq 0, \langle v_1,v-v_1 \rangle>0, 
\langle v_1^2 \rangle  \langle v^2 \rangle>\langle v_1,v \rangle^2
\end{equation}
(see \cite[Def. 1.2 and Prop. 1.3]{Y}).

\begin{lem}\label{lem:nowall}
Let $v=(r,\xi,a)$ be an isotropic Mukai vector such that $(\xi \cdot H) \geq 0$.
Then there is no wall, that is, ${\cal M}_{(\beta,\omega)}(v)$ is independent
of the choice of $(\beta,\omega)$.
In particular 
$$
{\cal M}_{(\beta,\omega)}(v)=\{E[k] \mid E \in {\cal M}_H(v) \},
$$
where
$k=0,1$ according as $r \geq 0$ or $r<0$.  
\end{lem}

\subsection{Totally semi-stable walls}\label{subsect:tss-wall}

\begin{defn}
A wall $W$ is totally semi-stable if
${\cal M}_{\sigma_+}(v) \cap {\cal M}_{\sigma_-}(v) =\emptyset$,
where $\sigma_\pm$ are in the adjacent two chambers. 
\end{defn}

We set
$$
{\cal I}_1:=\{ v_1 \mid \langle v_1^2 \rangle=0, \langle v,v_1 \rangle=1 \}.
$$

\begin{NB}
A chamber ${\cal C}$ is a connected component of
$$
{\cal H} \setminus \bigcup_{v_1 \in {\cal I}_1} W_{v_1} 
$$
where $v_1$ defines a totally semi-stable wall.
By \cite[Lem. 3.19]{Y},
$W_{v_1}$ does not intersect with other walls.
\end{NB}

\begin{prop}[{\cite[Thm. 5.3.5]{MYY:2018}, see also \cite[Lem. 5.3.4]{MYY:2018} and \cite[Thm. 4.2]{NY:2020}}]\label{prop:tot-wall}
$W$ is a totally semi-stable wall if and only if $W=W_{v_1}$ $(v_1 \in {\cal I}_1)$.
In particular there is no totally semi-stable walls if $v$ is not primitive.
\end{prop}

For a totally semi-stable wall $W_{v_1}$ defined by $v_1 \in {\cal I}_1$,
we have a decomposition
\begin{equation}\label{eq:tss-eq}
v=\ell v_1+v_2, \; \langle v_1^2 \rangle=\langle v_2^2 \rangle=0,\; 
\langle v_1,v_2 \rangle=1. 
\end{equation}

\begin{rem}\label{rem:codim0}
Let ${\cal C}_\pm$ be adjacent chambers of $W_{v_1}$ $(v_1 \in {\cal I}_1)$.
Then there are (contravariant) Fourier-Mukai transforms
$\Psi_\pm:{\bf D}(X) \to {\bf D}(Y)$ which induce isomorphisms
$$
\Psi_\pm:{\cal M}_{(\beta_\pm,\omega_\pm)}(v) \to  {\cal M}_{H_Y}(1,0,-\ell)
$$
where $Y=M_H(v_1)$, $(\beta_\pm,\omega_\pm) \in {\cal C}_\pm$ and $H_Y$ is an ample divisor on $Y$
(cf. \cite[Thm. 4.9]{MYY2}).
In particular 
$\Psi_+ \circ \Psi_-^{-1}$ induces an isomorphism
${\cal M}_{(\beta_-,\omega_-)}(v) \cong {\cal M}_{(\beta_+,\omega_+)}(v)$.
\end{rem}

Then Proposition \ref{prop:tot-wall} and Proposition \ref{prop:isom} imply that
the following corollary holds.
\begin{cor}\label{cor:non-primitive}
Assume that $v$ is not primitive.
Then $\Phi(E)[k]$ is a stable sheaf for a general $E \in {\cal M}_H(v)$, where $k=0$ or $1$ according as 
$a>0$ or $a \leq 0$.
\end{cor}

By \cite[Thm. 3.8, Prop. 5.14]{YY2}, we have the following result.
\begin{prop}\label{prop:isom}
\begin{enumerate}
\item[(1)]
$\Phi[1]$ induces an isomorphism
$$
{\cal M}_{(0,tH)}(r,\xi,a) \to {\cal M}_{(0,(nt)^{-1} \widehat{H})}(-a,\widehat{\xi},-r).
$$
\item[(2)]
Assume that $t \ll 1$.
\begin{enumerate}
\item
If $a \leq 0$, then
${\cal M}_{(0,(nt)^{-1} \widehat{H})}(-a,\widehat{\xi},-r)={\cal M}_{\widehat{H}}(-a,\widehat{\xi},-r)$.
\item
If $a>0$, then
${\cal M}_{(0,(nt)^{-1} \widehat{H})}(-a,\widehat{\xi},-r)$ 
consists of $E^{\vee}[1]$, $E \in {\cal M}_{\widehat{H}}(a,-\widehat{\xi},r)$.
\end{enumerate}
\end{enumerate}
\end{prop}

\section{Preservation of stability.}\label{sect:Preserve}

\subsection{Wall crossing along a line.}\label{subsect:line}

We shall study totally semi-stable walls along the line
\begin{equation}\label{eq:line}
{\cal L}:=\{(0,tH) \mid t>0 \} \subset {\cal H}.
\end{equation}
We set 
$$
v=(r,\xi,a),\; \xi=dH+D, D \in H^\perp.
$$
Then 
\begin{equation}
Z_{(0,tH)}(v)=(rt^2n-a)+2ndt \sqrt{-1},\; n=\frac{(H^2)}{2}.
\end{equation}
Hence 
\begin{equation}\label{eq:mu}
-\frac{\mathrm{Re}Z_{(0,tH)}(v)}{\mathrm{Im}Z_{(0,tH)}(v)}=
\frac{a-rt^2 n}{2ndt}.
\end{equation}

Assume that $r>0$ and $(\xi \cdot H)>0$.
Let $W_u$ be a totally semi-stable wall defined by $u \in {\cal I}_1$.
Thus 
\begin{equation}\label{eq:u}
v=\ell_1 v_1+\ell_2 v_2, \; \langle v_1^2 \rangle=\langle v_2^2 \rangle=0,\; 
\langle v_1,v_2 \rangle=1,\; \{ \ell_1,\ell_2 \}=\{ \ell,1 \},\; u \in \{v_1,v_2 \}. 
\end{equation}
Assume that $(0,t_0 H) \in W_u$.
We take $(0,t_\pm H)$ from adjacent chambers. Thus
\begin{equation}
t_0-\epsilon<t_-<t_0 <t_+ <t_0+\epsilon, \;\;(0<\epsilon \ll 1).
\end{equation} 
Let us study ${\cal M}_{(0,t_\pm H)}(v)$.
For $E \in {\cal M}_{(0,t_+ H)}(v)$, we have an exact triangle
\begin{equation}\label{eq:HNF-}
E_1 \to E \to E_2 \to E_1[1]
\end{equation}
which is the Harder-Narasimhan filtration with respect to
$\sigma_{(0,t_- H)}$-semistability,
where
$E_1 \in {\cal M}_{(0,t_0 H)}(\ell_1 v_1)$ and
$E_2 \in {\cal M}_{(0,t_0 H)}(\ell_2 v_2)$.
We set 
$$
v_i=(r_i,\xi_i,a_i),\;\xi_i=d_i H+D_i,\; D_i \in H^\perp.
$$
Then $d_1,d_2>0$ and
$$
\frac{a_1 d-ad_1}{r_1 d-rd_1}=t^2 \frac{(H^2)}{2}.
$$

\begin{NB}
$$
\frac{r_1 d-rd_1}{a_1 d-ad_1}=(nt)^{-2} \frac{(\widehat{H}^2)}{2}.
$$
\end{NB}



\begin{lem}\label{lem:d_1}
For the exact triangle \eqref{eq:HNF-}, we have the following.
\begin{enumerate}
\item[(1)]
$rd_1-r_1 d<0$ and $ad_1-a_1 d<0$.
\item[(2)]
$r_1>0$. In particular $E_1$ is a semi-homogeneous vector bundle.
\end{enumerate}
\end{lem}

\begin{proof}
By \eqref{eq:mu}, we get   
$$
\frac{a-rt^2 n}{2ntd}-\frac{a_1-r_1 t^2 n}{2ntd_1}
=\frac{(ad_1-a_1 d)-(rd_1-r_1 d)t^2 n}{2nt dd_1}>0
$$
for $t>t_0$.
Hence $rd_1-r_1 d<0$, which implies $ad_1-a_1 d<0$.
In particular $r_1>0$ and hence $E_1$ is a semi-homogeneous vector bundle.
\end{proof}

We next study the opposite chamber.
\begin{prop}\label{prop:opp}
For $E \in {\cal M}_{(0,t_- H)}(v)$, we have an exact triangle 
\begin{equation}\label{eq:opp}
E_2 \to E \to E_1 \to E_2[1]
\end{equation}
which is the Harder-Narasimhan filtration with respect to $\sigma_{(0,t_+ H)}$-semistability, where
$E_1 \in {\cal M}_{(0,t_0 H)}(\ell_1 v_1)$ and
$E_2 \in {\cal M}_{(0,t_0 H)}(\ell_2 v_2)$.
Then we have the following
\begin{enumerate}
\item
$E \not \in \Coh(X)$ if and only if $r_2<0$.   
In this case, $H^{-1}(E)[1] \cong E_2$ and $H^0(E) \cong E_1$.
\item
$E \in \Coh(X)$ is not torsion free if and only if $r_2=0$.
In this case, $E_2$ is the torsion submodule of $E$.
\item
$E \in \Coh(X)$ is a locally free sheaf if and only if $r_2>0$.
In this case
\eqref{eq:opp} is the Harder-Narasimhan filitration of $E$
with respect to
$H$.
\end{enumerate}
\end{prop}

\begin{proof}
By Lemma \ref{lem:d_1},  $E_1$ is a semi-homogeneous vector bundle.
Hence $H^{-1}(E_2) \cong H^{-1}(E)$ and 
we have an exact sequence
\begin{equation}
0 \to H^0(E_2) \to H^0(E) \to H^0(E_1) \to 0.
\end{equation}
Hence we have the following.
\begin{enumerate}
\item
$E \not \in \Coh(X)$ if and only if $H^0(E_2)=0$ and
$H^{-1}(E_2)$ is a semi-homogeneous vector bundle. 
In this case $H^{-1}(E) \cong H^{-1}(E_2)$ and $H^0(E) \cong H^0(E_1)$.  
\item
$E \in \Coh(X)$ is not torsion free if and only if $E_2$ is a torsion sheaf.
In this case $E_2$ is the torsion subsheaf of $E$.
\item
$E \in \Coh(X)$ is a locally free sheaf if and only if $E_2$ is a semi-homogeneous vector bundle.
\end{enumerate}
We note that $E_1$ and $E_2$ are Gieseker semi-stable locally free sheaves
with respect to $H$ for the case (iii).  Then Lemma \ref{lem:d_1} implies
\eqref{eq:opp} is the Harder-Narasimhan filitration of $E$
with respect to
$H$. Therefore our claims hold.
\end{proof}

\begin{lem}\label{lem:H_pm}
In the notation of \eqref{eq:u}, assume that $r_2>0$.
Then there are ample divisors $H_\pm$ such that
${\cal M}_{(0,t_\pm H)}(v)={\cal M}_{H_\pm}(v)$.
\end{lem}

\begin{proof}
Since
$((r_2 \xi_1-r_1 \xi_2)^2)<0$, there is an ample divisor $H_1$ such that
$((r_2 \xi_1-r_1 \xi_2) \cdot H_1)=0$.
We take ample divisors $H_\pm$ from a small neighborhood of $H_1$  
such that 
$$
((r_2 \xi_1-r_1 \xi_2) \cdot H_+)<0<((r_2 \xi_1-r_1 \xi_2) \cdot H_-).
$$
Then ${\cal M}_{H_\pm}(v)={\cal M}_{(0,t_\pm H)}(v)$. 
\end{proof}

By Remark \ref{rem:codim0},
we have a contravariant equivalence
$\Lambda:
{\bf D}(X) \to {\bf D}(X)$
such that $\Lambda(E_1)=E_1$, $\Lambda(E_2)=E_2$
and $\Lambda$ induces an isomorphism
\begin{equation}
{\cal M}_{(0,t_+ H)}(v){ \cong \cal M}_{(0,t_- H)}(v).
\end{equation}
Proposition \ref{prop:opp} shows that the birational maps in \cite{Y:abel}
are the birational maps constructed by wall-crossing along the line 
${\cal L}$ in \eqref{eq:line}.

The following proposition will play an important role in the proof of
our main result (Theorem \ref{thm:stability}).

\begin{prop}\label{prop:t_0}
There are real numbers $t_1 \geq t_2 \geq 0$ such that
\begin{enumerate}
\item
if $t>t_1$, then 
$E$ is a torsion free sheaf for a general $E \in {\cal M}_{(0,tH)}(v)$ and
\item
if $t_1>t>t_2$, then
$E$ is a coherent sheaf with torsions for a general
$E \in {\cal M}_{(0,tH)}(v)$ and
\item
if $t<t_2$, then
 $H^{-1}(E) \ne 0$ for a general $E \in {\cal M}_{(0,tH)}(v)$.
\end{enumerate}
\end{prop}

\begin{proof}
Let $W_{u}$ be a totally semi-stable wall and $(0,\mu H) \in W_{u}$. 
Assume that a general $E \in {\cal M}_{\sigma_{(0,tH)}}(v)$ is a coherent sheaf but 
is not torsion free for $(\mu+\epsilon >t>\mu)$.
Then $E$ fits in an exact sequence
$$
0 \to T \to E \to E/T \to 0
$$
where $T$ is the torsion subsheaf of $E$.
Let   
$$
E_1 \to E \to E_2 \to E_1[1]
$$
be the exact triangle
which is the Harder-Narasimhan filtration of $E$ 
with respect to $\sigma_{(0,tH)}$-semistability $(t<\mu)$. 
Then $v(E_1)=\ell_1 v_1$ and $v(E_2)=\ell_2 v_2$ in the notation of \eqref{eq:u}.
By Lemma \ref{lem:d_1}, $E_1$ is a semi-homogeneous bundle.
Assume that $H^{-1}(E_2)=0$.
Then $H^0(E_2)$ is a semi-homogeneous sheaf and
we have an exact sequence
\begin{equation}\label{eq:T}
0 \to H^0(E_1) \to H^0(E) \to H^0(E_2) \to 0.
\end{equation}
Hence we have an injective homomorphism $T \to H^0(E_2)$, which
shows that $H^0(E_2)$ is a torsion sheaf. 
Hence all members $E \in {\cal M}_{\sigma_{(0,tH)}}(v)$ $(t<\mu)$
are 
coherent sheaves with torsions or $H^{-1}(E) \ne 0$.

Assume that a general $E \in {\cal M}_{\sigma_{(0,tH)}}(v)$
satisfies $H^{-1}(E) \ne 0$ 
for a chamber $(\mu+\epsilon>t>\mu)$.
Then $E$ 
fits in an exact triangle
$$
H^{-1}(E)[1] \to E \to H^0(E) \to H^{-1}(E)[2].
$$
We have an exact triangle   
$$
E_1 \to E \to E_2 \to E_1[1]
$$
which is the Harder-Narasimhan filtration of $E$ 
with respect to $\sigma_{(0,tH)}$-semistability $(t<\mu)$. 
Then $v(E_1)=\ell_1 v_1$ and $v(E_2)=\ell_2 v_2$ in the notation of \eqref{eq:u}.
By Lemma \ref{lem:d_1}, $E_1$ is a semi-homogeneous vector bundle.
Since $H^{-1}(E) \ne 0$, we see that $E_2[-1] \in \Coh(X)$ and   
we have an exact sequence
\begin{equation}\label{eq:E_2}
0 \to H^{-1}(E) \to H^{-1}(E_2) \to H^0(E_1) \to H^0(E) \to 0. 
\end{equation}
Hence all members $E \in {\cal M}_{(0,tH)}(v)$ $(t<\mu)$
satisfies $H^{-1}(E) \ne 0$.
Therefore there are real numbers $t_1 \geq t_2 \geq 0$
satisfying (i), (ii) and (iii).
\end{proof}

\begin{rem}
\begin{enumerate}
\item[(1)]
For the torsion subsheaf $T$ of $E$, 
we set $v(T)=\ell_1' v_1'$, $v(E/T)=\ell_2' v_2'$, 
where 
$\{\ell_1',\ell_2' \}=\{\ell,1\}$.
Since the extension \eqref{eq:T} does not split,
$T \to H^0(E)$ is not isomorphic.
Hence $\ell_1'=1$ and $\ell_1=\ell$.
\item[(2)]
For the exact sequence \eqref{eq:E_2}, we get
\begin{enumerate}
\item
$\rk H^{-1}(E_2) > \rk H^{-1}(E)$ or
\item
$\rk H^{-1}(E_2) = \rk H^{-1}(E)$ and 
$(c_1(H^{-1}(E_2)) \cdot H) > (c_1(H^{-1}(E)) \cdot H)$.
\end{enumerate}
\end{enumerate}
\end{rem}

\begin{lem}\label{lem:Gieseker}
Assume that $t_1+\epsilon>t>t_1>0$, where $\epsilon$ is sufficiently small positive number.
Then there is an ample divisor $L$ such that
${\cal M}_{(0,tH)}(v) \cap {\cal M}_{L}(v) \ne \emptyset$.
\end{lem}

\begin{proof}
If ${\cal M}_{(0,tH)}(v) \cap {\cal M}_{H}(v) \ne \emptyset$, then
obviously the claim holds.
So we assume that ${\cal M}_{(0,tH)}(v) \cap {\cal M}_{H}(v)= \emptyset$.
Then there is a number $t'$ such that there is no totally semi-stable wall containing $(0,tH)$
with $t_1<t<t'$ and $(0,t'H)$ is contained in a totally semi-stable wall $W$.
Then the claim follows from Lemma \ref{lem:H_pm}. 
\end{proof}

\begin{NB}
$$
0 \to E_1 \to E \to E_2 \to 0
$$
Assume that $\ell_1=\ell$.
Let $\Phi_{X \to X'}^{{\bf E}^{\vee}}:{\bf D}(X) \to {\bf D}(X')$ be 
a Fourier-Mukai transform such that
$\Phi_{X \to X'}^{{\bf E}^{\vee}}(E)[1]$ is a torsion free sheaf of rank 1
and 
$$
0 \to \Phi_{X \to X'}^{{\bf E}^{\vee}}(E)[1] \to \Phi_{X \to X'}^{{\bf E}^{\vee}}(E_2)[1]
\to \Phi_{X \to X'}^{{\bf E}^{\vee}}(E_1)[2] \to 0.
$$
Let $F$ be a subsheaf of $E$ with 
$\frac{(c_1(F) \cdot (H+\eta))}{\rk F}> \frac{(c_1(E) \cdot (H+\eta))}{\rk E}$.
We may assume that 
$\frac{(c_1(F) \cdot H)}{\rk F}= \frac{(c_1(E) \cdot H)}{\rk E}$
We note that $F \cap E_1$ and $E_1/F \cap E_1$ are $\mu$-semi-stable sheaves.
Hence $v(E_1 \cap F)=kv_1$. If $0<k \leq \ell$, then 
we see that
$\frac{(c_1(F) \cdot (H+\eta))}{\rk F}< \frac{(c_1(E) \cdot (H+\eta))}{\rk E}$.
Hence $F \subset E_1=0$. Then $F \to E_2$ is isomorphic, which is a contradiction.
\end{NB}

For the wall crossing along the line 
$$
{\cal L}':=\{(0,t' \widehat{H}) \mid t'>0 \},
$$
Proposition \ref{prop:t_0} is restated as follows.
 
\begin{prop}\label{prop:t_0'}
\begin{enumerate}
\item[(1)]
Assume that $a>0$. For the Mukai vector $v'=(a,\widehat{\xi},r)$,
there are real numbers $t_1' \geq t_2' \geq 0$ such that
\begin{enumerate}
\item
if $t'>t_1'$, then a general $F \in {\cal M}_{(0,t' \widehat{H})}(v')$
is a torsion free sheaf and
\item
if $t_1'>t'>t_2'$, then a general $F \in {\cal M}_{(0,t' \widehat{H})}(v')$
is a coherent sheaf with torsions and
\item
if $t'<t_2'$, then
 $H^{-1}(F) \ne 0$ for a general $F \in {\cal M}_{(0,t' \widehat{H})}(v')$.
\end{enumerate}
\item[(2)]
Assume that $a<0$. For the Mukai vector $v'=(-a,\widehat{\xi},-r)$,
there are real numbers $t_1' \geq t_2' \geq 0$ such that
\begin{enumerate}
\item
if $t'>t_1'$, then a general $F \in {\cal M}_{(0,t' \widehat{H})}(v')$ 
is a torsion free sheaf and
\item
if $t_1'>t'>t_2'$, then a general $F \in {\cal M}_{(0,t' \widehat{H})}(v')$
is a coherent sheaf with torsions and
\item
if $t'<t_2'$, then
 $H^{-1}(F) \ne 0$ for a general $F \in {\cal M}_{(0,t' \widehat{H})}(v')$.
\end{enumerate}
\end{enumerate}
\end{prop}

\begin{rem}
Assume that $a=0$.
For the Mukai vector $v'=(0,\widehat{\xi},r)$,
there are real numbers $t_2' \geq 0$ such that
\begin{enumerate}
\item
if $t'>t_2'$, then a general $F \in {\cal M}_{(0,t' \widehat{H})}(v')$
is a purely 1-dimensional sheaf and
\item
if $t'<t_2'$, then
 $H^{-1}(F) \ne 0$ for a general $F \in {\cal M}_{(0,t' \widehat{H})}(v')$.
\end{enumerate}
\end{rem}

\subsection{Proof of Theorem \ref{thm:main}.}\label{subsect:proof}

Let $E_1$ and $E_2$ be semi-stable objects in \eqref{eq:HNF-} with Mukai vectors
\begin{equation}
v(E_i)=v_i=(r_i,\xi_i,a_i),\;(i=1,2).
\end{equation}
We first study the Fourier-Mukai transforms $\Phi(E_1),\Phi(E_2)$ of $E_1,E_2$
by using Lemma \ref{lem:semihom}.

\begin{lem}\label{lem:a>0}
Assume that $a >0$.
Then $a_1>0$,
$\Phi(E_1) \in
{\cal M}_{\widehat{H}}(\ell_1(a_1,-\widehat{\xi_1},r_1))$
and the following claims hold.
\begin{enumerate}
\item[(1)]
Assume that $r_2<0$.
Then $a_2>0$ and $\Phi(E_2) \in {\cal M}_{\widehat{H}}(\ell_2(a_2,-\widehat{\xi_2},r_2))$.
\item[(2)]
Assume that $r_2=0$.
Then one of the following holds.
\begin{enumerate}
\item
$a_2>0$ and $\Phi(E_2) \in {\cal M}_{\widehat{H}}(\ell_2(a_2,-\widehat{\xi_2},r_2))$.
\item
$a_2=0$, 
$(\xi_2^2)=0$, $(\xi_1 \cdot \xi_2)=1$
and $\Phi(E_2)[1] \in {\cal M}_{\widehat{H}}(\ell_2(0,\widehat{\xi_2},0))$.
In particular $X$ is a product of elliptic curves and $\xi$ is primitive.
\begin{NB}
$(\xi \cdot \xi_i)=\ell_{1-i}$.
Hence $\xi$ is primitive. 
\end{NB}
\end{enumerate}
\item[(3)]
Assume that $r_2>0$. Then the following claims hold.
\begin{enumerate}
\item
If $a_2 > 0$, then
$\Phi(E_2) \in {\cal M}_{\widehat{H}}(\ell_2(a_2,-\widehat{\xi_2},r_2))$.
\item
If $a_2 \leq 0$, then 
$\Phi(E_2)[1] \in {\cal M}_{\widehat{H}}(\ell_2(-a_2,\widehat{\xi_2},-r_2))$. 
\end{enumerate}
\end{enumerate}
\end{lem}

\begin{proof}
By Lemma \ref{lem:d_1}, $r_1>0$ and $a_1 d-d_1 a>0$.
Hence we get $a_1>d_1 a/d \geq 0$.
Then $(\xi_1^2)=2r_1 a_1>0$ implies $\xi_1$ is ample.
Hence we have $\Phi(E_1) \in {\cal M}_{\widehat{H}}(\ell_1(a_1,-\widehat{\xi_1},r_1))$.

(1) Assume that $r_2<0$.
(i) If $a_2<0$, then $(\xi_2^2)=2r_2 a_2>0$ implies $\xi_2$ is ample.
Hence 
$$
\langle v_1,v_2 \rangle=(\xi_1 \cdot \xi_2)-r_1 a_2-r_2 a_1 \geq 3,
$$
which is a contradiction.
(ii) If $a_2=0$ and $\xi_2 \ne 0$, then $\xi_2$ is effective and
$$
\langle v_1,v_2 \rangle=(\xi_1 \cdot \xi_2)-r_2 a_1 \geq 2,
$$
 which is a contradiction.
\begin{NB}
$(\xi_2 \cdot H)=d_2(H^2)>0$.
\end{NB}
If $a_2=\xi_2=0$, then $v_2=(-1,0,0)$.
In this case, $(d,a)=(\ell_1 d_1,\ell_1 a_1)$ implies that
$a_1 d-d_1 a=0$.
Therefore this case does not occur too.
(iii) If $a_2>0$, then $(\xi_2^2)=2r_2 a_2<0$ and
$\Phi(E_2) \in {\cal M}_{\widehat{H}}(\ell_2(a_2,-\widehat{\xi_2},r_2))$.

(2)
We assume that $r_2=0$. Then $\xi_2$ is an effective divisor with $(\xi_2^2)=0$.
If $a_2 < 0$, then 
we have $(\xi_1 \cdot \xi_2) \geq 0$ and $-r_1 a_2 \geq 0$.
If $(\xi_1 \cdot \xi_2)=0$, then $r_1 a_1=(\xi_1^2)/2=0$, which is a contradiction.
Hence $(\xi_1 \cdot \xi_2)>0$.
Since 
$$
1=\langle v_1,v_2 \rangle=(\xi_1 \cdot \xi_2)-r_1 a_2,
$$
we see that
 $(\xi_1 \cdot \xi_2 )=1$ and $r_1 a_2=0$,
which is a contradiction.
Therefore $a_2 \geq 0$.
If $a_2=0$, then 
$X$ is a product of elliptic curves
and $\Phi(E_2)[1] \in {\cal M}_{\widehat{H}}(\ell_2(0,\widehat{\xi_2},0))$.

(3)
Assume that $r_2 > 0$.
If $a_2 \geq 0$, then $(\xi_2^2) =2r_2 a_2 \geq 0$.
Since $(\xi_2 \cdot H)>0$, 
$a_2>0$ and $\Phi(E_2) \in {\cal M}_{\widehat{H}}(\ell_2(a_2,-\widehat{\xi_2},r_2))$ or
$a_2=0$ and
$\Phi(E_2)[1] \in {\cal M}_{\widehat{H}}(\ell_2(0,\widehat{\xi_2},-r_2))$ is a torsion sheaf.
If $a_2<0$, then
$(\xi_2^2)=2r_2 a_2 < 0$.
Thus $\Phi(E_2)[1] \in {\cal M}_{\widehat{H}}(\ell_2(-a_2,\widehat{\xi_2},-r_2))$. 
\end{proof}

\begin{rem}
If $a_2=0$, then $\Phi(E)$ is a two-term complex of locally free sheaves and
$\Phi^1(E)$ is a torsion sheaf.
Hence $\Phi(E)^{\vee}$ is a coherent sheaf with a torsion.
 \end{rem}

\begin{NB}
If $X$ does not contain an elliptic curve, then 
$a_2 \ne 0$.
\end{NB}

\begin{lem}\label{lem:a<0}
Assume that $a \leq 0$.
Then
$a_2<0$,
$\Phi(E_2)[1] \in {\cal M}_{\widehat{H}}(\ell_2(-a_2,\widehat{\xi_2},-r_2))$
and the following claims hold.
\begin{enumerate}
\item[(1)]
Assume that $r_2<0$.
Then $a_1<0$ and $\Phi(E_1)[1] \in 
{\cal M}_{\widehat{H}}(\ell_1(-a_1,\widehat{\xi_1},-r_1))$.
\item[(2)]
Assume that $r_2=0$. Then one of the following holds.
\begin{enumerate}
\item
If $a_1=0$, then there is an elliptic curve $C$ and
$v_1=(1,k_1 C,0)$, $v_2=(0,k_2 C,-1)$, where $k_1,k_2 \in {\Bbb Z}_{>0}$.
\item
If $a_1<0$, then $\Phi(E_1)[1] \in {\cal M}_{\widehat{H}}(\ell_1(-a_1,\widehat{\xi_1},-r_1))$.
\end{enumerate}
\item[(3)]
\begin{NB}
This case does not occur for the proof of Theorem \ref{thm:stability} (2).
\end{NB}
Assume that $r_2>0$.
\begin{enumerate}
\item
 If $a_1 \leq 0$, then $\Phi(E_1)[1] \in {\cal M}_{\widehat{H}}(\ell_1(-a_1,\widehat{\xi_1},-r_1))$.
\item
If $a_1>0$, then $\Phi(E_1) \in  {\cal M}_{\widehat{H}}(\ell_1(a_1,-\widehat{\xi_1},r_1))$.
\end{enumerate}
\end{enumerate}
\end{lem}

\begin{proof}
We note that $E_1$ is a semi-homogeneous bundle (Lemma \ref{lem:d_1} (2)). 
Since 
$$
\ell_1(rd_1-r_1 d)+\ell_2(rd_2-r_2 d)=0,
$$
Lemma \ref{lem:d_1} implies
$rd_2-r_2 d>0$ and $ad_2-a_2 d>0$.
Hence $a_2<0$ and we see that
$\Phi(E_2)[1] \in {\cal M}_{\widehat{H}}(\ell_2(-a_2,\widehat{\xi_2},-r_2))$.

(1) Assume that $r_2<0$.
Then $(\xi_2^2)=2r_2 a_2>0$ implies $\xi_2$ is ample.
(i) If $a_1>0$, then $\xi_1$ is ample.
Since 
$$
\langle v_1,v_2 \rangle=(\xi_1 \cdot \xi_2)-r_1 a_2-r_2 a_1 \geq 3,
$$
this case does not occur.
(ii) If $a_1=0$, then $\xi_1$ is effective and
$$
\langle v_1,v_2 \rangle=(\xi_1 \cdot \xi_2)-r_1 a_2 \geq 2.
$$
Hence this case does not occur too.
(iii) If $a_1<0$, then $\Phi(E_1)[1] \in {\cal M}_{\widehat{H}}(\ell_1(-a_1,\widehat{\xi_1},-r_1))$.

(2) Assume that $r_2=0$.
Then $\xi_2$ is effective with $(\xi_2^2)=0$.
(i) If $a_1>0$, then $\xi_1$ is ample, which implies
$$
\langle v_1,v_2 \rangle=(\xi_1 \cdot \xi_2)-r_1 a_2 \geq 2.
$$
Hence this case does not occur.
(ii) If $a_1=0$, then $\xi_1$ is effective.
Since 
$$
1=\langle v_1,v_2 \rangle=(\xi_1 \cdot \xi_2)-r_1 a_2 \geq -r_1 a_2>0,
$$
$(\xi_1 \cdot \xi_2)=0$ and
$v_1=(1,\xi_1,0)$, $v_2=(0,\xi_2,-1)$.
Since $\xi_i$ $(i=1,2)$ are effective divisor with $(\xi_i^2)=0$, 
there is an elliptic curve $C$ and
$\xi_i=k_i C$ $(k_i \in {\Bbb Z}_{>0})$.
(iii) If $a_1<0$, then $\Phi(E_1)[1] \in {\cal M}_{\widehat{H}}(\ell_1(-a_1,\widehat{\xi_1},-r_1))$.

(3) Assume that $r_2>0$.
(i) If $a_1<0$, then $\Phi(E_1)[1] \in {\cal M}_{\widehat{H}}(\ell_1(-a_1,\widehat{\xi_1},-r_1))$.
(ii) If $a_1= 0$, then $\Phi(E_1)[1]$ is a torsion sheaf.
(iii) If $a_1>0$, then $\Phi(E_1) \in {\cal M}_{\widehat{H}}(\ell_1(a_1,-\widehat{\xi_1},r_1))$.
\end{proof}

\begin{thm}\label{thm:stability}
Let $v=(r,\xi,a)$ be a Mukai vector such that $r > 0$ and $(\xi \cdot H)>0$, where 
$H$ is an ample divisor on $X$.
\begin{enumerate}
\item[(1)]
Assume that $a>0$. 
If $X$ is not a product of elliptic curves or $\xi$ is not primitive,
then there are ample divisors $L \in \NS(X)$ and $L' \in \NS(\widehat{X})$ such that 
$\Phi(E)^{\vee}$ is stable with respect to $L'$ for a general
$E \in {\cal M}_L(v)$.
\item[(2)]
Assume that $a \leq 0$ and
$v \ne (\ell,kC,-1), (1,kC,-\ell)$, where 
$C$ is an elliptic curve and $k \geq \ell+1$. 
Then there are ample divisors $L \in \NS(X)$ and $L' \in \NS(\widehat{X})$ such that 
$\Phi(E)[1]$ is stable with respect to $L'$ for a general
$E \in {\cal M}_L(v)$.
\end{enumerate}
\end{thm}

\begin{proof}
\begin{NB}
We may assume that $H$ is a general ample divisor with respect to $v$.
\end{NB}
For the family of stability conditions
$\sigma_{(0,tH)}$, let $t_1 \geq t_2$ be non-negative numbers in Proposition \ref{prop:t_0}.
If $t_1+\epsilon> t>t_1>0$, then Lemma \ref{lem:Gieseker} implies
${\cal M}_{(0,tH)}(v) \cap {\cal M}_L(v) \ne \emptyset$ for an ample divisor $L$.

(1)
Assume that $t_1>t_2$ and take a member $E \in {\cal M}_{(0,tH)}(v)$ ($t_1+\epsilon> t>t_1$).
Let   
\begin{equation}\label{eq:HNFt1}
E_1 \to E \to E_2 \to E_1[1]
\end{equation}
be the exact triangle
which is the Harder-Narasimhan filtration of $E$ 
with respect to $\sigma_{(0,tH)}$-semistability $(t_1-\epsilon<t<t_1)$. 
Then $v(E_1)=\ell_1 v_1$ and $v(E_2)=\ell_2 v_2$ in the notation of \eqref{eq:u}.
By the proof of Proposition \ref{prop:t_0},
$E_2$ is a torsion semi-homogeneous sheaf.
By Lemma \ref{lem:a>0}, $\Phi(E_1)$ is a semi-homogeneous bundle.
Since $X$ is not a product of elliptic curves or $\xi$ is not primitive,
$\Phi(E_2)$
is also a semi-homogeneous bundle (Lemma \ref{lem:a>0} (2)).
Hence we have an exact sequence
$$
0 \to \Phi(E_2)^{\vee} \to \Phi(E)^{\vee} \to \Phi(E_1)^{\vee} \to 0.
$$
By Lemma \ref{lem:Gieseker}, 
$\Phi(E)^{\vee} \in {\cal M}_{L'}(a,\widehat{\xi},r)$ for an ample divisor $L'$ 
on $\widehat{X}$.

Assume that $t_1=t_2>0$
and take a member $E \in {\cal M}_{(0,tH)}(v)$ ($t_1+\epsilon> t>t_1$).
Let   
\begin{equation}\label{eq:HNFt0}
E_1 \to E \to E_2 \to E_1[1]
\end{equation}
be the exact triangle
which is the Harder-Narasimhan filtration of $E$ 
with respect to $\sigma_{(0,tH)}$-semistability $(t_2-\epsilon <t<t_2)$. 
Then $v(E_1)=\ell_1 v_1$ and $v(E_2)=\ell_2 v_2$ in the notation of \eqref{eq:u}.
By the proof of Proposition \ref{prop:t_0},
$E_2[-1]$ is a semi-homogeneous vector bundle.
By Lemma \ref{lem:a>0},
$\Phi(E_1) $ and $\Phi(E_2)$
are semi-homogeneous vector bundles
and we have an exact sequence
$$
0 \to \Phi(E_2)^{\vee} \to \Phi(E)^{\vee} \to \Phi(E_1)^{\vee} \to 0.
$$
By Lemma \ref{lem:Gieseker}, 
$\Phi(E)^{\vee} \in {\cal M}_{L'}(a,\widehat{\xi},r)$ for an ample divisor $L'$
on $\widehat{X}$.

Assume that $t_1=t_2=0$.
We take a general $E \in {\cal M}_H(v)$ such that  
$E \in  {\cal M}_{(0,tH)}(v)$ ($\epsilon> t>0$).
Then  $\Phi(E)^{\vee}$ is a stable sheaf with respect to $\widehat{H}$
by Proposition \ref{prop:isom}.

(2)
Assume that $t_1>t_2$ and take a member $E \in {\cal M}_{(0,tH)}(v)$ ($t_1+\epsilon> t>t_1$)
fitting in an exact triangle
$$
E_1 \to E \to E_2 \to E_1[1]
$$
as in \eqref{eq:HNFt1}.
Since $E_2$ is a torsion semi-homogeneous sheaf,
by using Lemma \ref{lem:a<0},
we see that $\Phi(E_1)[1]$ and $\Phi(E_2)[1]$
are semi-homogeneous bundles
fitting in an exact sequence
$$
0 \to \Phi(E_1)[1] \to \Phi(E)[1] \to \Phi(E_2)[1] \to 0.
$$
By Lemma \ref{lem:Gieseker}, 
$\Phi(E)[1] \in {\cal M}_{L'}(-a,\widehat{\xi},-r)$ for an ample divisor $L'$
on $\widehat{X}$.

Assume that $t_1=t_2>0$
and take a member $E \in {\cal M}_{(0,tH)}(v)$ ($t_1+\epsilon> t>t_1$)
fitting in an exact triangle
$$
E_1 \to E \to E_2 \to E_1[1]
$$
as in \eqref{eq:HNFt0}.
Since $r_2<0$, by using Lemma \ref{lem:a<0},
we see that $\Phi(E_1)[1] $ and $\Phi(E_2)[1]$
are semi-homogeneous bundles
and we have an exact sequence
$$
0 \to \Phi(E_1)[1] \to \Phi(E)[1] \to \Phi(E_2)[1] \to 0.
$$
By Lemma \ref{lem:Gieseker}, 
$\Phi(E)[1] \in {\cal M}_{L'}(-a,\widehat{\xi},-r)$ for an ample divisor $L'$
on $\widehat{X}$.

Assume that $t_1=t_2=0$. We take a general $E \in {\cal M}_H(v)$ such that  
$E \in  {\cal M}_{(0,tH)}(v)$ ($\epsilon> t>0$).
Then  $\Phi(E)[1]$ is a stable sheaf with respect to $\widehat{H}$
by Proposition \ref{prop:isom}.
\end{proof}

\begin{NB}
Relation of parameter:
For $v'=(-a,\widehat{\xi},-r)$, we take $t_1' \geq t_2' \geq 0$. Then
$\frac{1}{t_2}>\frac{1}{t_1}>t_1' \geq t_2' \geq 0$.
In particular if there is no totally semistable wall in $\Amp(X)_{\Bbb R}$, then
$t_1'=t_2'=0$. 
\end{NB}

\begin{rem}
We can rewrite Theorem \ref{thm:stability} in terms of the non-negative numbers $t_1' \geq t_2'$ 
in Proposition \ref{prop:t_0'}.
We note that $\frac{1}{n t_2'} \geq \frac{1}{n t_1'} \geq 0$, where $\frac{1}{n t_2'}=\infty$ if $t_2'=0$.
By Proposition \ref{prop:isom},
we have an isomorphism 
${\cal M}_{(0,\frac{1}{nt'}H)}(r,\xi,a) \cong {\cal M}_{(0,t' \widehat{H})}(-a,\widehat{\xi},-r)$
by $\Phi[1]$.
Then the statements of Theorem \ref{thm:stability} imply that
$\frac{1}{nt_1'}>t_1$.
\end{rem}

\begin{rem}\label{rem:exc}
We shall treat the exceptional cases in Theorem \ref{thm:stability}.
\begin{enumerate}
\item[(1)]
Assume that $X$ is a product of elliptic curves.
Let $C_1,C_2$ be elliptic curves in $X$ such that $(C_1 \cdot C_2)=1$.
\begin{enumerate}
\item
Assume that $v=(\ell r_1,\ell C_1+(\ell r_1 a_1+1)C_2,\ell a_1)$.
Then $u=(0,C_2,0)$ defines a totally semi-stable wall $W_u$.
By the irreducibility of $C_2$ and the proof of Proposition \ref{prop:t_0}, we see that $(0,t_1 H) \in W_u$.
In this case, we get $t_1=\frac{1}{nt_1'}$ and 
$\Phi(E)^{\vee}$ is not torsion free for any $E \in {\cal M}_H(v)$.
\item
Assume that $v=(r_1,C_1+(r_1 a_1 +\ell)C_2,a_1)$.
Similarly $u=(0,C_2,0)$ defines a totally semi-stable wall $W_u$ with $(0,t_1 H) \in W_u$.
Hence we get $t_1=\frac{1}{nt_1'}$ and 
$\Phi(E)^{\vee}$ is not torsion free for any $E \in {\cal M}_H(v)$.
\end{enumerate}
\begin{NB}
$\frac{a_2 d_1-a_1 d_2}{r_2 d_1-r_1 d_2}=\frac{a_1}{r_1}$.
Hence $t_1^2 n=a_1/r_1$ and a general $E$ is stable with respect to
an ample divisor $L$. 
\end{NB}
\item[(2)]
Assum  that $X$ contains an elliptic curve $C$. 
\begin{enumerate}
\item
Assume that $v=(\ell,k C,-1)$ $(k \geq \ell+1)$.
We take positive integers $k_1, k_2$ such that
$k=\ell k_1+k_2$, $k_2 \leq \ell$. 
For $u_i=(0,(k_2+i\ell)C,-1)$ with $0 \leq i<k$,
we have a decomposition
$$
v=\ell (1,(k_1-i)C,0)+u_i, 
$$
which implies we have a totally semi-stable wall $W_{u_i}$.
Then we see that $(0,t_1 H) \in W_{u_0}$ and $t_1=\frac{1}{nt_1'}$.
\begin{NB}
Since $t^2 n=-\frac{k_2+i \ell -k}{r(k_2+i \ell)}$, $t^2 n \geq \frac{k_2-k}{rk_2}$.
\end{NB} 
Since $\rk E=\ell$, ${\cal M}(v)$ is irreducible,
Hence $\Phi(E)[1]$ is a coherent sheaf with torsions for all
$E \in {\cal M}_H(v)$.
\item
Assume that
$v=(1,k C,-\ell)$ $(k \geq \ell+1)$.
Then $u=(0,C,-1)$ defines a totally semi-stable wall $W_u$
containing $(0,t_1 H)$. 
In this case $\Phi(E)[1]$ is a coherent sheaf with torsions for all $E \in {\cal M}_H(v)$.
\begin{NB}
In generat, we have a decomposition
$v=\ell v_1+\ell_2 v_2$, $v_1=(1,k_1 C,0), v_2=(0,k_2 C,-1)$.
If $k_2>\ell_1$, then $E$ has a torsion.
\end{NB}
\end{enumerate}
\begin{NB}
$\frac{a_2 d_1-a_1 d_2}{r_2 d_1-r_1 d_2}=\frac{k_1}{k_2}$.
Hence $t_1^2 n=k_1/k_2$ and a general $E$ is stable with respect to
an ample divisor $L$. 
\end{NB}
\end{enumerate}
\end{rem}

\begin{cor}\label{cor:stability}
Let $v=(r,\xi,a)$ be a primitive Mukai vector such that $r>0$ and $(\xi \cdot H)>0$.
\begin{enumerate}
\item[(1)]
Assume that $a>0$ and $\langle v^2 \rangle \geq 2r, 2a$.
If $X$ is not a product of elliptic curves or $\xi$ is not primitive, then
$\Phi(E)^{\vee}$ is stable with respect to $\widehat{H}$ for a general
$E \in {\cal M}_H(v)$.
\item[(2)]
Assume that $(\xi^2)>0$ and $a< 0$.
Then $\Phi(E)[1]$ is stable with respect to $\widehat{H}$ for a general
$E \in {\cal M}_H(v)$.
\end{enumerate}
\end{cor}

\begin{proof}
(1)
By using Proposition \ref{prop:chamber} and Theorem \ref{thm:stability},
the claim follows.
(2)
We note that 
$$
\langle v^2 \rangle=(\xi^2)-2ra > -2ra.
$$
By using Proposition \ref{prop:chamber} and Theorem \ref{thm:stability},
the claim follows.
\end{proof}

\begin{prop}\label{prop:wBN-L}
Let $v=(r,\xi,a)$ be a Mukai vector such that $r>0$, $(\xi \cdot H)>0$ and $\langle v^2 \rangle \geq 0$.
Then there is an ample divisor $L$ such that the weak Brill-Noether property holds for
${\cal M}_L(v)$.  
\end{prop}

\begin{proof}
By Theorem \ref{thm:stability} and Remark \ref{rem:exc},
there is an ample divisor $L$ such that 
\begin{enumerate}
\item
if $a>0$, then $\Phi(E)^{\vee}$ is a sheaf for a general $E \in {\cal M}_L(v)$, and
\item
if $a \leq 0$, then $\Phi(E)[1]$ is a sheaf for a general $E \in {\cal M}_L(v)$.
\end{enumerate}
Therefore the weak Brill-Noether property holds.
\end{proof}

\begin{rem}[{cf. \cite[Prop. 3.6]{CNY2}}]
\begin{enumerate}
\item
Assume that $v \ne e^\eta (r,0,-1)$ $(\eta \in \NS(X))$.
Then there is a $\mu$-stable locally free sheaf $E$ with $v(E)=v$ (\cite[Prop. 3.5]{KY}).
By the Grothendieck-Serre duality, we have
$\Phi_{X \to \widehat{X}}^{{\cal P}^{\vee}}(E)=\Phi_{X \to \widehat{X}}^{{\cal P}}(E^{\vee})^{\vee}[2]$.
Hence if $(\xi \cdot H)<0$, then
we can apply Theorem \ref{thm:stability} to $E^{\vee}$.
\item
Assume that $v = e^\eta (r,0,-1)$ $(\eta \in \NS(X))$.
Then ${\cal M}_H(v)$ consists of non-locally free sheaves for any $H$.
If $(\eta^2)<0$, then 
there is an ample divisor $L' \in \NS(\widehat{X})$ such that
$\Phi(E)[1]$ is stable with respect to $L'$ for a general $E \in {\cal M}_H(v)$.
If $(\eta^2) \geq 0$ and $(\eta \cdot H) \leq 0$, then
$\Phi^1(E) \ne 0$ and $\Phi^2(E) \ne 0$.
\begin{NB}
If $(\eta \cdot H)=0$, then $\eta=0$. In this case $\Phi(E)^{\vee}[-1]$ is a stable sheaf 
with the Mukai vector $(1,0,-r)$. 
\end{NB}
\end{enumerate}
\end{rem}

\begin{NB}
\begin{rem}
If $v$ is not primitive, then 
\end{rem}
\end{NB}

\begin{NB}
We don't need the following anymore since $t<t_1$.
\begin{lem}
Assume that $0<k \leq \ell$.
Let $E$ be a general extension fitting in an exact sequence
\begin{equation}\label{eq:nonsplit}
0 \to E_1 \to E \to E_2 \to 0
\end{equation}
where $E_1 \in {\cal M}_H(\ell,0,0)$ and $E_2 \in {\cal M}_H(0,kC,-1)$.
Then $E$ is torsion free.
\end{lem}

\begin{proof}
Assume that $E$ is not torsion free. Let $T$ be the torsion subsheaf and
$F:=E/T$. Then we have an exact sequence
$$
0 \to T \to E \to F \to 0
$$
such that 
\begin{equation}
\begin{split}
v(T)= & \ell_1 w_1,\; w_1=(0,\xi_1,a_1),\\
v(F)=& \ell_2 w_2,\; w_2=(r_2,\xi_2,a_2),
\end{split}
\end{equation}
$\langle w_1^2 \rangle=\langle w_2^2 \rangle=0$, $\langle w_1,w_2 \rangle=1$ and $\{\ell_1,\ell_2\}=\{1,\ell\}$.
Since $T \to E \to E_2$ is injective,
$\xi_1=k_1 C$ with $0<k_1 \leq k$.
Since \eqref{eq:nonsplit} is a non-split sequence,
$k_1<k$.
By $\ell=\langle v,v(T) \rangle=-\ell \ell_1 a_1$, $\ell_1=1$ and $a_1=-1$.
Then $\ell \xi_2=(k-k_1)C$, which is a contradiction.
Therefore $E$ is torsion free.
\end{proof}

\begin{lem}
For a general $Z=\{ x_1,x_2,...,x_\ell \}$,
$I_Z$ fits in an exact sequence
$$
0 \to {\cal O}_X(-\sum_{i=1}^\ell C_i) \to I_Z \to \oplus_{i=1}^\ell {\cal O}_{C_i}(-x_i) \to 0
$$ 
where $x_i \in C_i$ and the algebraic equivalence class of $C_i$ are $C$.  
Hence a general $E \in {\cal M}_H(1,kC,-\ell)$ fits in an exact sequence 
\begin{equation}
0 \to E_1 \to E \to E_2 \to 0
\end{equation}
$E_1 \in {\cal M}_H(1,(k-\ell)C,0)$ and $E_2 \in {\cal M}_H(0,\ell C,-\ell)$.
\end{lem}

\end{NB}

\subsection{Wall crossing for ${\cal M}_H(r,dH,a)$.}\label{subsect:dH}

\begin{NB}
The choice of polarization is important.
In particular if the polarization is very close to the boundary of the ample cone,
we need to change the polarization.
We show that $H=\xi$ is a good polarization if $\xi$ is ample.
\end{NB}
In this subsection, we treat the case where $\xi$ is ample and the polarization is $\xi$.
Thus we assume that $v=(r,dH,a)$. 
Let us study wall crossing for $v=(r,dH,a)$ along ${\cal L}$.
We keep the notation in subsection \ref{subsect:proof}.
Thus $E_1$ and $E_2$ be semi-stable objects in \eqref{eq:HNF-}
with Mukai vectors
\begin{equation}
v(E_i)=v_i=(r_i,\xi_i,a_i),\;(i=1,2).
\end{equation}

\begin{NB}
\begin{equation}\label{eq:tss-cond}
v=\ell_1 v_1+\ell_2 v_2,\;\langle v_1,v_2 \rangle=1,\; \langle v_1^2 \rangle=\langle v_2^2 \rangle=0,\;
 \{\ell_1,\ell_2 \}=\{ \ell,1 \}
\end{equation}

We set
\begin{equation}
v_1=(r_1,\xi_1,a_1),\;v_2=(r_2,\xi_2,a_2).
\end{equation}
\end{NB}
Then we have
\begin{equation}\label{eq:dH1}
r=\ell_1 r_1+\ell_2 r_2,\;dH=\ell_1 \xi_1+\ell_2 \xi_2,\; a=\ell_1 a_1+\ell_2 a_2
\end{equation}
and
\begin{equation}\label{eq:dH2}
(\xi_1^2)=2r_1 a_1,\;(\xi_2^2)=2r_2 a_2,\; (\xi_1 \cdot \xi_2)=r_1 a_2+r_2 a_1+1.
\end{equation}

By Lemma \ref{lem:d_1}, $r_1>0$ and
\begin{equation}\label{eq:d1}
r_2 d(\xi_1 \cdot H)-r_1 d(\xi_2 \cdot H)<0,\;
a_2 d(\xi_1 \cdot H)-a_1 d(\xi_2 \cdot H)<0.
\end{equation}

\begin{lem}\label{lem:Z-slope}
\begin{equation}
\begin{split}
r_1 d(\xi_2 \cdot H)-r_2 d (\xi_1 \cdot H)=&
(r_1 \ell_1+r_2 \ell_2)(r_1 a_2-r_2 a_1)+(r_1 \ell_1-r_2 \ell_2)\\
a_1 d(\xi_2 \cdot H)-a_2 d (\xi_1 \cdot H)=&
(a_1 \ell_1+a_2 \ell_2)(r_2 a_1-r_1 a_2)+(a_1 \ell_1-a_2 \ell_2).
\end{split}
\end{equation}
\end{lem}

\begin{proof}
By \eqref{eq:dH1} and \eqref{eq:dH2}, we see that
\begin{equation}
\begin{split}
r_1 d(\xi_2 \cdot H)-r_2 d (\xi_1 \cdot H)=&
r_1(\ell_1 (\xi_1 \cdot \xi_2)+\ell_2 (\xi_2^2))-r
_2(\ell_1 (\xi_1^2)+\ell_2 (\xi_1 \cdot \xi_2))\\
=& (r_1 \ell_1-r_2 \ell_2)(\xi_1 \cdot \xi_2)+r_1 \ell_2 (\xi_2^2)-r_2 \ell_1 (\xi_1^2)\\
=& (r_1 \ell_1-r_2 \ell_2)(r_2 a_1+r_1 a_2+1)+2r_2 \ell_2 r_1 a_2-2r_1 \ell_1 r_2 a_1\\
=& (r_1 \ell_1+r_2 \ell_2)(r_1 a_2-r_2 a_1)+(r_1 \ell_1-r_2 \ell_2).
\end{split}
\end{equation}

In the same way, we see that
\begin{equation}
\begin{split}
a_1 d(\xi_2 \cdot H)-a_2 d (\xi_1 \cdot H)=&
a_1(\ell_1 (\xi_1 \cdot \xi_2)+\ell_2 (\xi_2^2))-
a_2(\ell_1 (\xi_1^2)+\ell_2 (\xi_1 \cdot \xi_2))\\
=& (a_1 \ell_1-a_2 \ell_2)(\xi_1 \cdot \xi_2)+a_1 \ell_2 (\xi_2^2)-a_2 \ell_1 (\xi_1^2)\\
=& (a_1 \ell_1-a_2 \ell_2)(r_2 a_1+r_1 a_2+1)+2a_2 \ell_2 a_1 r_2-2a_1 \ell_1 a_2 r_1\\
=& (a_1 \ell_1+a_2 \ell_2)(r_2 a_1-r_1 a_2)+(a_1 \ell_1-a_2 \ell_2).
\end{split}
\end{equation}
\end{proof}

\begin{lem}\label{lem:r_1a_2-r_2a_1=0}
Assume that $r_1,r_2,a_1,a_2>0$.
Then $r_1 a_2-r_2 a_1=0$.
\end{lem}

\begin{proof}
We use Lemma \ref{lem:Z-slope}. 
If $r_1 a_2-r_2 a_1<0$, then we see that
\begin{equation}
\begin{split}
r_1 d(\xi_2 \cdot H)-r_2 d (\xi_1 \cdot H)
\leq & -(r_1 \ell_1+r_2 \ell_2)+(r_1 \ell_1-r_2 \ell_2)=-2r_2 \ell_2<0\\
\end{split}
\end{equation}
Hence
By \eqref{eq:d1},
this case does not occur.

If $r_1 a_2-r_2 a_1>0$, then 
\begin{equation}
\begin{split}
a_1 d(\xi_2 \cdot H)-a_2 d (\xi_1 \cdot H)
\leq & -(a_1 \ell_1+a_2 \ell_2)+(a_1 \ell_1-a_2 \ell_2)=-2a_2 \ell_2<0.
\end{split}
\end{equation}
%
By \eqref{eq:d1},
this case does not occur either.
Hence $r_1 a_2-r_2 a_1=0$.
\end{proof}

\begin{lem}\label{lem:tss-L}
Assume that there is a totally semi-stable wall on ${\cal L}$.
Then there are relatively prime integers $p,q$ and integers
$k_1>0$, $k_2 \geq 0$ such that 
\begin{equation}\label{eq:wall-dH}
\begin{split}
&v=((\ell_1 k_1+\ell_2 k_2)p,\ell_1 \xi_1+\ell_2 \xi_2,(\ell_1 k_1+\ell_2 k_2)q),\;\{\ell_1,\ell_2\}=\{\ell,1\}\\
& (\xi_i^2)=2k_i^2 pq,\; (\xi_1 \cdot \xi_2)=2k_1 k_2 pq+1.
\end{split}
\end{equation}
\end{lem}

\begin{proof}
We note that
$$
r_1 \ell_1+r_2 \ell_2=r>0,\;
a_1 \ell_1+a_2 \ell_2=a>0.
$$
By Lemma \ref{lem:a>0},
$a_1>0$ and
we have 5 cases to treat.

(1) $r_2<0$ and $a_2>0$.
In this case, $r_1 a_2-r_2 a_1 >0$. Hence
\begin{equation}
\begin{split}
& a_1 d(\xi_2 \cdot H)-a_2 d(\xi_1 \cdot H) \leq -2a_2 \ell_2<0.
\end{split}
\end{equation}
Therefore this case does not occur.

(2)
$r_2=a_2=0$.
In this case,
$(\xi_2^2)=0$ and $(\xi_1 \cdot \xi_2)=1$.
Hence 
$$
v=(\ell_1 r_1,\ell_1 \xi_1+\ell_2 \xi_2,\ell_1 a_1). 
$$
Then $v$ is of the form \eqref{eq:wall-dH}, where   
$r_1=k_1 p, a_1=k_1 q$ and $k_2=0$.

(3)
$r_2=0$ and $a_2>0$.
In this case, $r_1 a_2-r_2 a_1>0$. Hence
\begin{equation}
\begin{split}
& a_1 d(\xi_2 \cdot H)-a_2 d(\xi_1 \cdot H) \leq -2a_2 \ell_2<0.
\end{split}
\end{equation}
Therefore this case does not occur. 

(4)
$r_2>0$ and $a_2 \leq 0$.
In this case, $r_1 a_2-r_2 a_1<0$.
Hence
\begin{equation}
\begin{split}
& r_1 d(\xi_2 \cdot H)-r_2 d(\xi_1 \cdot H) \leq -2r_2 \ell_2<0.\\
\end{split}
\end{equation}
Therefore this case does not occur. 

(5)
$r_2$ and $a_2>0$. In this case,
Lemma \ref{lem:r_1a_2-r_2a_1=0} implies that
$r_1 a_2-r_2 a_1=0$.
Then there are relatively prime integers $p,q$ and positive integers
$k_i$ ($i=1,2$) such that $r_i=k_i p$ and $a_i=k_i q$.
Then we see that
$$
(\xi_i^2)=2k_i^2 pq,\; (\xi_1 \cdot \xi_2)=2k_1 k_2 pq+1.
$$
Therefore the claim holds.
\end{proof}

\begin{NB}
Assume that $v=(r,0,-1)e^{kH}$, where $rk=d$.
If $k<0$, then $\Phi_{X \to \widehat{X}}^{{\cal P}^{\vee}}(E)$
is not a sheaf 
see \cite[Prop. 3.5]{CNY2}.
\end{NB}

\begin{prop}
Assume that $d>0$ and $a>0$.
\begin{enumerate}
\item[(1)]
If $v$ is not written as in \eqref{eq:wall-dH}, then
$\Phi(E)^{\vee} \in {\cal M}_{\widehat{H}}(a,d \widehat{H},r)$
 for a general $E \in {\cal M}_H(r,dH,a)$.
\item[(2)]
Assume that 
\begin{enumerate}
\item
$X$ is not a product of two elliptic curves and $\gcd(r,a) \leq \ell$ or 
\item $d > r$.
\end{enumerate}
Then
$v$ is not written as in \eqref{eq:wall-dH}.
\end{enumerate}
\end{prop}

\begin{proof}
(1) is a consequence of Lemma \ref{lem:tss-L}.

(2)
Assume that $v$ is written as in \eqref{eq:wall-dH}.
Since $\gcd(r,a)=\ell_1 k_1+\ell_2 k_2$, if $k_2 \ne 0$, then $\gcd(r,a) > \ell$.
We note that
$$
d((\ell_2 \xi_2-\ell_1 \xi_1)\cdot H)=2pq (\ell_2 k_2+\ell_1 k_1)(\ell_2 k_2-\ell_1 k_1).
$$
Since 
$$
d(\xi_1 \cdot H) \equiv \ell_2 \mod 2pq,\;
d(\xi_2 \cdot H) \equiv \ell_1 \mod 2pq,
$$
and one of $\ell_i$ is 1,
we get $\gcd(d,2pq)=1$.
Since $v$ is primitive, we also get 
$\gcd(d,\ell_1 k_1+\ell_2 k_2)=1$.
Hence 
we get $d \mid (\ell_2 k_2-\ell_1 k_1)$.
Hence $d \leq |\ell_2 k_2-\ell_1 k_1| \leq \ell_1 k_1+\ell_2 k_2 \leq r$.
Therefore the claim holds.
\end{proof}

\begin{lem}\label{lem:tssH}
Assume that $H$ is a totally semi-stable wall with respect to $v=(r,dH,a)$.
\begin{enumerate}
\item
[(1)]
We have a decomposition of $v$ such that
\begin{equation}
\begin{split}
& v=\ell_1 v_1+\ell_2 v_2,\;\langle v_1,v_2 \rangle=1,\; \langle v_1^2 \rangle=\langle v_2^2 \rangle=0,\;
 \{\ell_1,\ell_2 \}=\{ \ell,1 \},\\
& v_i=(r_i,\xi_i,a_i),\;r_i>0,\; (i=1,2),\;\; r_2(\xi_1 \cdot H)-r_1(\xi_2 \cdot H)=0, \,\,
r_1 a_2=r_2 a_1.
\end{split}
\end{equation}
\item[(2)]
Let $E$ be a $\mu$-semi-stable sheaf with respect to $H$ and
$v(E)=v$. Then
$E$ is a Gieseker semi-stable locally free sheaf which is $S$-equivalent to
$\oplus_{i=1}^2 (\oplus_{j=1}^{\ell_i} F_{ij})$,
where $F_{ij} \in {\cal M}_H(v_i)$.
\end{enumerate}
\end{lem}

\begin{proof}
(1)
If there is no $\mu$-stable sheaf $E$ with $v(E)=v$, then
we have a decomposition of $v$:
\begin{equation}
v=\ell_1 v_1+\ell_2 v_2,\;\langle v_1,v_2 \rangle=1,\; \langle v_1^2 \rangle=\langle v_2^2 \rangle=0,\;
 \{\ell_1,\ell_2 \}=\{ \ell,1 \}
\end{equation}
where 
\begin{equation}
v_i=(r_i,\xi_i,a_i),\;r_i>0,\; (i=1,2),\;\; r_2(\xi_1 \cdot H)-r_1(\xi_2 \cdot H)=0.
\end{equation}
By Lemma \ref{lem:Z-slope}, we see that $r_1 a_2=r_2 a_1$ and $r_1 \ell_1=r_2 \ell_2$.
Hence we have a description \eqref{eq:wall-dH} of $v$ such that $\ell_1 k_1=\ell_2 k_2$.

(2)
For simplicity, we assume that $\ell_1=\ell$.
For a $\mu$-semi-stable sheaf $E$ with $v(E)=kv_1+v_2$ ($k>0$), 
we shall show that there is an exact sequence
\begin{equation}\label{eq:JHF}
0 \to E_1 \to E \to E_2 \to 0
\end{equation}
such that $E_1$ and $E_2$ are $\mu$-semi-stable with
$\{v(E_1),v(E_2) \}=\{v_1,(k-1)v_1+v_2 \}$.

Let ${\cal E}$ be a universal family on $X \times M_H(v_1)$.
Assume that 
$$
\Hom({\cal E}_{|X \times \{ y \}},E)=\Ext^2({\cal E}_{|X \times \{ y \}},E)=0
$$
for all $y \in M_H(v_1)$, then 
$\Phi_{X \to M_H(v_1)}^{{\cal E}^{\vee}}(E)[1]$ is a line bundle, and hence
its Mukai vector is isotropic.
Therefore there is a non-zero homomorhism
$\varphi:{\cal E}_{|X \times \{ y \}} \to E$ or a non-zero homomorphism
$\psi:E \to {\cal E}_{|X \times \{ y \}}$.
Obviously $\varphi$ is injective and $\coker \varphi$ is $\mu$-semi-stable with
the Mukai vector $(k-1)v_1+v_2$.
In this case we set $E_1={\cal E}_{|X \times \{ y \}}$ and $E_2:=\coker \varphi$.
For the homomorphism $\psi$,
$\coker \psi$ is a 0-dimensional sheaf and $\ker \psi$ is $\mu$-semi-stable.
We prove that $\coker \psi=0$.
We set $v(\coker \psi)=(0,0,n)$. Then 
$v(\ker\psi)=(k-1)v_1+v_2+(0,0,n)$ and 
we see that
$$
0 \leq \langle v(\ker \psi)^2 \rangle=2(k-1)-2n((k-1)r_1+r_2) \leq -2nr_2.
$$
Hence $n=0$, and $\psi$ is surjective.
So we set $E_1:=\ker \psi$ and $E_2:={\cal E}_{|X \times \{ y \}}$.
Then we get a desired exact sequence \eqref{eq:JHF}.
By the induction on $k$, we get (2).
\end{proof}

\begin{lem}\label{lem:wall-dH2}
We assume that $a \leq 0$.
Then there is no totally semi-stable wall on ${\cal L}$.
In particular $\Phi(E)[1] \in {\cal M}_{\widehat{H}}(-a,d\widehat{H},-r)$
for a general $E \in {\cal M}_H(r,dH,a)$.
\end{lem}

\begin{proof}
By Lemma \ref{lem:a<0}, we have two possibilities:

(1)
$r_2 \leq 0$ and $a_1,a_2<0$.
In this case $r_1 a_2-r_2 a_1<0$.
Since $a_1 d(\xi_2 \cdot H)-a_2 d(\xi_1 \cdot H) \leq 2a_1 \ell_1<0$,
this case does not occur by \eqref{eq:d1}.

(2)
$r_2>0$ and $a_2<0$. 
In this case, we treat by cases according as the sign of $a_1$.
\begin{enumerate}
\item
If $a_1 \geq 0$, then
$r_1 a_2-r_2 a_1<0$.
Hence 
\begin{equation}
\begin{split}
& r_1 d(\xi_1 \cdot H)-r_2 d(\xi_2 \cdot H) \leq -2r_2 \ell_2<0.\\
\end{split}
\end{equation}
Therefore this case does not occur.
\item
If $a_1<0$, then 
$\langle v,v_i \rangle=(dH \cdot \xi_i)-ra_i-r_i a \geq 3$.
Hence this case does not occur.
\end{enumerate}
\end{proof}

\begin{thm}\label{thm:wBN}
Assume that $v=(r,dH,a)$.
\begin{enumerate}
\item[(1)] 
If $d \geq 0$,
the weak Brill-Noether property holds.
Thus there is $E \in {\cal M}_H(v)$ such that $E$ has at most one nonzero cohomology
group.
\item[(2)]
If $d<0$, then
the weak Brill-Noether property holds unless $v=(r,0,-1)e^{kH}$, where
$k$ is a negative integer.
\end{enumerate}
\end{thm}


\begin{proof}
(1)
If $d=0$, then we see that $a \leq 0$ and $H^0(X,E)=H^2(X,E)=0$ for a general 
$E \in {\cal M}_H(v)$.
Therefore we assume that $d>0$. Then
$H^2(X,E)=0$ for all $E \in {\cal M}_H(v)$ by the Serre duality and the stability of $E$.
Assume that $a>0$.
If there is no totally semi-stable wall on ${\cal L}$, 
then $H^1(X,E)=0$ for a general $E \in {\cal M}_H(v)$.
If there is a totally semi-stable wall on ${\cal L}$, then
Lemma \ref{lem:tss-L} implies
$E$ fits in an exact sequence
$$
0 \to E_1 \to E \to E_2 \to 0
$$
where $E_1$ and $E_2$ are semi-homogeneous sheaves with
$v(E_i)=\ell_i v_i$.
Since $a_1>0$ and $a_2 \geq 0$, we see that
$H^1(X,E)=0$ for a general $E$. 
If $a \leq 0$, then the claim is a consequence of Lemma \ref{lem:wall-dH2}.

(2)
Assume that $v \ne (r,0,-1)e^{kH}$.
If there is a $\mu$-stable locally free sheaf $E \in {\cal M}_H(v)$, then
$E^{\vee}$ is also a $\mu$-stable locally free sheaf.
Hence the claim follows from (1) and the Serre duality
$H^1(X,E) \cong H^1(X,E^{\vee})^{\vee}$.

We first assume that $H$ is general with respect to $H$.
In this case \cite[Prop. 3.5]{KY} implies that 
a general $E \in {\cal M}_H(v)$ is a $\mu$-stable locally free sheaf, and hence the claim holds.

We next assume that $H$ is not general.
\begin{NB}
If there is no $\mu$-stable sheaf $E$ with $v(E)=v$, then
we have a decomposition of $v$:
\begin{equation}
v=\ell_1 v_1+\ell_2 v_2,\;\langle v_1,v_2 \rangle=1,\; \langle v_1^2 \rangle=\langle v_2^2 \rangle=0,\;
 \{\ell_1,\ell_2 \}=\{ \ell,1 \}
\end{equation}
where 
\begin{equation}
v_i=(r_i,\xi_i,a_i),\;r_i>0,\; (i=1,2),\;\; r_2(\xi_1 \cdot H)-r_1(\xi_2 \cdot H)=0.
\end{equation}
By Lemma \ref{lem:Z-slope}, we see that $r_1 a_2=r_2 a_1$ and $r_1 \ell_1=r_2 \ell_2$.
Hence we have a description \eqref{eq:wall-dH} of $v$ such that $\ell_1 k_1=\ell_2 k_2$.
In this case, every $E \in {\cal M}_H(v)$ is $S$-equivalent to $\oplus_{i=1}^2 (\oplus_{j=1}^{\ell_i} F_{ij})$,
where $F_{ij} \in M_H(v_i)$.
\end{NB}
If there is no $\mu$-stable locally free $E$ with $v(E)=v$, then
$H$ is on a totally semi-stable wall (\cite[Prop. 3.6]{KY}). Applying Lemma \ref{lem:tssH},
we get $E^{\vee} \in {\cal M}_H(v^{\vee})$ for all $E \in {\cal M}_H(v)$, and we get (2) in this case. 
\end{proof}

As a corollary, we get the following result which
shows that the claim \cite[Thm. 0.2]{Y:aCM-abel} holds for
any polarized abelian surface.

\begin{cor}
Let $v:=(r,dH,a) \in H^*(X,{\Bbb Z})_{\alg}$ be a Mukai vector with $r>0$.
Then there is a $H$-semi-stable aCM sheaf $E$ with $v(E)=v$ if and only if
$v=(r,d_0 H,a)e^{d'H}$ with
$d_0^2 n/r \geq a \geq 0$ and $r/2 \geq |d_0|$.
\end{cor}

\begin{NB}
\begin{equation}
\begin{split}
v=&((\ell_1 k_1+\ell_2 k_2)p,\ell_1 \xi_1+\ell_2 \xi_2,(\ell_1 k_1+\ell_2 k_2)q)\\
=& \ell_1(k_1 p,\xi_1,k_1 q)+\ell_2 (k_2 p,\xi_2,k_2 q)
\end{split}
\end{equation}
where 

\begin{equation}
\gcd(p,q)=1 (p>0),\;
(\xi_i^2)=2k_i^2 pq,\; (\xi_1 \cdot \xi_2)=2k_1 k_2 pq+1.
\end{equation}

$p,q$ and $\ell_1 k_1+\ell_2 k_2$ are determined by $r$ and $a$.
\end{NB}

\begin{NB}
If $k_1 \ell_1=k_2 \ell_2$, then $v_1$ shows that $E$ is a properly semi-stable sheaf.
\end{NB}

\subsubsection{A relation to \cite{BMOY}.}
In \cite{BMOY}, we studied the movable cone of a fiber
of the albanese map $M_H(r,dH,a) \to X \times \widehat{X}$ by
studying Bridgeland walls,
where $d>0$ and $a<0$.
From the arguments, we can deduce a refinement of 
Lemma \ref{lem:wall-dH2} under the assumption $r \geq 2, d>0,a \geq 2$.

Let ${\cal I}$ be the set of isotropic Mukai vectors $u$
such that $\langle u,v \rangle=1,2$.
\begin{NB}
In \cite{BMOY}
we classified $u \in {\cal I}$ separating
$(0,rH,d(H^2))$ and $(-1,\tfrac{-a}{d(H^2)}H,0)$.
Thus we assume that
\begin{equation}\label{eq:condition1}
\begin{split}
0 > \langle u,(0,rH,d(H^2)) \rangle
\langle u,(-1,\tfrac{-a}{d(H^2)}H,0) \rangle 
=(r(H,\eta)-pd(H^2))(\tfrac{-a}{d(H^2)}(H,\eta)+q).
\end{split}
\end{equation}
It is the same as the classification of wall $W_u$ intersecting ${\cal L}$.
In particular there is no wall unless $X$ is a product of two elliptic curve and
$d=1$.
Under this condition we have 
\end{NB}
\begin{NB}
$$
-w_0^{\vee}=(0,rH,(H^2)d),\;
-w_1^{\vee}=(-(H^2)d,-aH,0).
$$
Hence $\langle w_i,u^{\vee} \rangle=-\langle -w\i^{\vee},u \rangle$. 
\end{NB}
Let $W$ be a wall and take $(0,t_\pm H)$ from adjacent chambers.
If 
$$
\dim ({\cal M}_{(0,t_\pm H)}(v) \setminus  {\cal M}_{(0,t_\mp H)}(v)) \geq
\dim {\cal M}_{(0,t_\pm H)}(v)-1,
$$
then there is $u \in {\cal I}$ such that $W=W_u$.
It is easy to see that
$(0,tH) \in W_u$ if and only if  
$$
\frac{t^2}{2}(H^2) \langle u,(0,rH,d(H^2)) \rangle+\langle u,(-d(H^2),-aH,0) \rangle=0.
$$
By the proof of
\cite[Lem. 6.3]{BMOY}, $W_u \cap {\cal L} =\emptyset$ for all $u \in {\cal I}$.

\begin{prop}
We have a birational map
$M_H(r,dH,a) \cdots \to M_{\widehat{H}}(-a,d\widehat{H},-r)$ which is defined
by $E \mapsto \Phi(E)[1]$
up to codimension 1,
unless (1) $v=(r,H,-1)$ or $v=(1,H,a)$ and (2)
there is a divisor $\eta$ such that
$(\eta,H)=1$, $(\eta^2)=0$.
\end{prop}

\begin{rem}
If $(H^2)=2$, then the polarized dual $(\widehat{X},\widehat{H})$ of $X$ is isomorphic to $(X,H)$.
Hence $\Phi[1]$ induces a birational involution of
$M_H(r,dH,-r)$.
\end{rem}

\section{Appendix.}\label{sect:appexdix}

\subsection{Relation of Mukai vectors defining totally semi-stable walls.}

Let $W$ and $W'$ be totally semi-stable walls with $(0,t_0 H) \in W$ and
$(0,t_0' H) \in W'$.
We assume that $t_0<t_0'$ and
there is no totally semi-stable wall containing $(0,tH)$ with $t_0<t<t_0'$.
Let 
\begin{equation}
v=\ell_1 v_1+\ell_2 v_2,\;\langle v_1,v_2 \rangle=1,\; \langle v_1^2 \rangle=\langle v_2^2 \rangle=0,\;
 \{\ell_1,\ell_2 \}=\{ \ell,1 \}
\end{equation}
and
\begin{equation}
v=\ell_1 ' v_1'+\ell_2' v_2',\;
\langle v_1',v_2' \rangle=1,\, \langle {v_1'}^2 \rangle=\langle {v_2'}^2 \rangle=0,\;
 \{\ell_1',\ell_2' \}=\{ \ell,1 \}
\end{equation}
be the corresponding decompositions of $v$. 
For a general $E$, we have an exact sequence
$$
0 \to E_1 \to E \to E_2 \to 0
$$
and
$$
0 \to E_2' \to E \to E_1' \to 0,
$$
where $E_i \in {\cal M}_{(0,tH)}(\ell_i v_i)$
and $E_i' \in {\cal M}_{(0,tH)}(\ell_i' v_i')$
are direct sum of stable objects.
We set 
\begin{equation}
\begin{split}
v_i=& (r_i,\xi_i,a_i),\; \xi_i=d_i H+D_i,\; D_i \in H^\perp, \; (i=1,2),\\
v_i'=& (r_i',\xi_i',a_i'),\; \xi_i'=d_i' H+D_i',\; D_i' \in H^\perp, \; (i=1,2).
\end{split}
\end{equation}
By our assumption $t_0<t_0'$, we have
$$
\frac{a_1' d_2'-a_2' d_1'}{r_1' d_2'-r_2' d_1'}>\frac{a_1 d_2-a_2 d_1}{r_1 d_2-r_2 d_1}.
$$
We shall prove that
\begin{equation}
\begin{split}
& r_1d_2-r_2 d_1>  r_1' d_2'-r_2' d_1'>0,\\
& a_1' d_2'-a_2' d_1' >  a_1 d_2-a_2 d_1>0.
\end{split}
\end{equation}

We note that 
$$
r_1 d-r d_1>0,\, r_1' d-r d_1'>0,\;
r_1>0,\;r_1'>0
$$
(Lemma \ref{lem:d_1}).

\begin{lem}\label{lem:dd'}
$r_1 d_1'-r_1' d_1>0$ and $r_2' d_2-r_2 d_2' \geq 0$.
Moreover $r_2' d_2-r_2 d_2' > 0$ unless $r_2=r_2'=0$.
\end{lem}

\begin{proof}
We first prove that
\begin{equation}\label{eq:E_2E_1'}
\Hom(E_2,E_1')=0.
\end{equation}
Assume that $r_2>0$. Then 
$\frac{d_2}{r_2}>\frac{d}{r}>\frac{d_1'}{r_1'}$,
Hence we get \eqref{eq:E_2E_1'}.
Assume that $r_2=0$. Then
$E_2$ is a torsion sheaf. Hence we get \eqref{eq:E_2E_1'}.
Assume that $r_2<0$.
Then $E_2=H^{-1}(E_2)[1]$. Hence we get \eqref{eq:E_2E_1'}.

Then we get an injective homomorphism
$$
0 \to \Hom(E,E_1') \to \Hom(E_1,E_1').
$$
Since $E_1$ and $E_1'$ are direct sum of $\mu$-stable vector bundles, 
we get $r_1 d_1'-r_1' d_1 \geq 0$.
If the equality holds, then Lemma \ref{lem:mu-s} implies $v_1=v_1'$, which is a contradiction.
Therefore $r_1 d_1'-r_1' d_1 >0$.

We next prove that $r_2' d_2-r_2 d_2' \geq 0$.
If $r_2'>0$ and
$r_2 \leq 0$, then $r_2' d_2-r_2 d_2'>0$.
Assume that $r_2'>0$ and $r_2>0$.
Then $\Hom(E_2',E_1)=0$ by $\frac{d_2'}{r_2'}>\frac{d_1}{r_1}$.
Hence 
we get an injective homomorphism
$$
0 \to \Hom(E_2',E) \to \Hom(E_2',E_2).
$$
Hence $r_2' d_2-r_2 d_2'>0$.

Assume that $r_2'=0$.
Since $E_2'$ is a torsion sheaf, $\Hom(E_2',E_1)=0$.
Hence
we get an injective homomorphism
$$
0 \to \Hom(E_2',E) \to \Hom(E_2',E_2).
$$
Then we get $r_2 \leq 0$. 
Therefore $r_2' d_2-r_2 d_2' =-r_2 d_2' \geq 0$ and the equality does not hold
unless $r_2=0$.

Assume that $r_2'<0$ and $r_2<0$.
Then $E_2[-1]$ and $E_2'[-1]$ are direct sum of $\mu$-stable vector bundles. 
We see that $H^{-1}(E)=H^{-1}(E_2')$ and $H^{-1}(E) \to H^{-1}(E_2)$ is injective.
Hence we get $\frac{d_2}{r_2}>\frac{d_2'}{r_2'}$.
Therefore $r_2' d_2-r_2 d_2' > 0$.
\end{proof}

\begin{NB}
$r_2 d_1-r_1 d_2<0$ and $r_2' d_1'-r_1' d_2'<0$.

Then 
\begin{equation}
\frac{d_2'}{r_2'}<\frac{d_2}{r_2},\;\frac{d_1}{r_1}<\frac{d_1'}{r_1'}.
\end{equation}
$r_1 d_2-r_2 d_1>r_1' d_2'-r_2' d_1'>0$.

$\Hom(E_2',E_2) \ne 0$. Hence $\langle v_2',v_2 \rangle<0$.
Since $\Hom(E_2,E_1')=0$, $\Hom(E_1,E_1') \ne 0$. Hence
$\langle v_1,v_1' \rangle<0$.

Assume that $r_2<0$ and $r_2'<0$.
Then 
$$
r_2' d_2-r_2 d_2'>0,\; r_1 d_1'-r_1' d_1>0.
$$
\end{NB}

\begin{lem}\label{lem:mu-s}
Let $F_1$ and $F_2$ be $\mu$-stable vector bundles with
$(c_1(F_1^{\vee} \otimes F_2) \cdot H)=0$.
Then $\Hom(F_1,F_2) \ne 0$ implies $F_1 \cong F_2$.
\end{lem}

\begin{proof}
For a non-trivial homomorphism $f:F_1 \to F_2$,
we see that $\ker f=0$ and $\coker f$ is of 0-dimensional.  
Then $\det f$ is an isomorphism, which implies $f$ is an isomorphism.
\end{proof}

\begin{prop}
$r_1d_2-r_2 d_1>r_1' d_2'-r_2' d_1'>0$.
\end{prop}

\begin{proof}
By using Lemma \ref{lem:dd'}, we get
\begin{equation}\label{eq:area}
\ell(r_1d_2-r_2 d_1)-\ell(r_1' d_2'-r_2' d_1')=\ell_1 \ell_1'(r_1 d_1'-r_1' d_1)+\ell_2 \ell_2'(d_2 r_2'-r_2 d_2')>0.
\end{equation}
\end{proof}

In order to prove the other inequality, we prepare some lemmas.

\begin{lem}\label{lem:chi}
Let $F_1$ and $F_2$ be stable objects with 
isotropic Mukai vectors.
If $\Hom(F_1,F_2) \ne 0$, then $\chi(F_1,F_2) \geq 0$.
\end{lem}

\begin{proof}
We note that $F_1$ and $F_2$ are semi-homogeneous sheaves up to shift.
Then the claim is a consequence of \cite[Prop. 4.4]{Y:abel}.
\end{proof}

\begin{lem}\label{lem:uu}
For isotropic Mukai vectors
$$
u_i=(r_i,\xi_i,a_i),\; \xi_i=d_i H+D_i,\; D_i \in H^\perp, \; (i=1,2),
$$
we have
\begin{equation}
d_1 d_2 \langle u_1,u_2 \rangle=-\frac{1}{2}((d_2 D_1-d_1 D_2)^2)+(d_2 r_1-d_1 r_2)(d_2 a_1-d_1 a_2).
\end{equation}
\end{lem}

\begin{NB}
\begin{lem}
$d_1 d_1'-a_1' d_1 \geq 0$ and $a_2' d_2-a_2 d_2' \geq 0$.
In particular 
$a_1' d_2'-a_2' d_1' \geq a_1 d_2-a_2 d_1>0$.
\end{lem}

\begin{proof}
By the proof of Lemma \ref{lem:dd'},
$\Hom(E_1',E_1) \ne 0$ and $\Hom(E_2',E_2) \ne 0$.
By Lemma \ref{lem:chi}, 
$\langle v_1,v_1' \rangle \leq 0$ and $\langle v_2,v_2' \rangle \leq 0$.
Applying Lemma \ref{lem:uu},
we get our first claim.
By using \eqref{eq:area}, we get the second claim.
\end{proof}
\end{NB}

\begin{lem}
\begin{enumerate}
\item[(1)]
If $a > 0$, then $a_1' d_1-a_1 d_1' > 0$ and $a_2 d_2'-a_2' d_2  \geq 0$.
\item[(2)]
If $a \leq 0$, then $a_1' d_1-a_1 d_1' \geq 0$ and
$a_2 d_2'-a_2' d_2 > 0$.
\item[(3)]
$a_1' d_2'-a_2' d_1' > a_1 d_2-a_2 d_1>0$.
\end{enumerate}
\end{lem}

\begin{proof}

(1) Asume that $a>0$.
Then Lemma \ref{lem:a>0} implies that
$\Phi(E_1)^{\vee}$ and $\Phi(E_1')^{\vee}$ are 
semi-homogeneous vector bundles 
with the Mukai vectors $\ell_1 (a_1,\widehat{\xi_1},r_1)$ and
$\ell_1' (a_1',\widehat{\xi_1'},r_1')$.
Since 
$$
\Hom(\Phi(E_1')^{\vee},\Phi(E_1)^{\vee}) \cong \Hom(E_1,E_1') \ne 0,
$$
by using Lemma \ref{lem:mu-s}, we get
$a_1' d_1-a_1 d_1' > 0$.

By $\Hom(E_2',E_2) \ne 0$ and Lemma \ref{lem:chi},
$\langle v_2,v_2' \rangle \leq 0$.
If $r_2 \ne 0$ or $r_2' \ne 0$, then 
Lemma \ref{lem:dd'} implies $r_2' d_2-r_2 d_2'>0$.
By using Lemma \ref{lem:uu}, we get 
$a_2 d_2'-a_2' d_2  \geq 0$.
If $r_2=r_2' =0$, then the stability of $E_2$ and $E_2'$ imply that
$a_2 d_2'-a_2' d_2 \geq 0$.

(2)
Assume that $a \leq 0$.
Then we get $a_2,a_2'<0$ by Lemma \ref{lem:a<0}.
In this case,
$\Phi(E_2)[1]$ and $\Phi(E_2')[1]$ are semi-homogeneous vector bundles
with Mukai vectors 
$\ell_2 (-a_2,\widehat{\xi_2},-r_2)$ and
$\ell_2' (-a_2',\widehat{\xi_2'},-r_2')$.
Since 
$$
\Hom(\Phi(E_2'),\Phi(E_2)) \cong \Hom(E_2',E_2) \ne 0,
$$
by using Lemma \ref{lem:mu-s}, we get
$a_2 d_2'-a_2' d_2 > 0$.

By $\Hom(E_1,E_1') \ne 0$ and Lemma \ref{lem:chi},
$\langle v_1,v_1' \rangle \leq 0$.
We also get $r_1' d_1-r_1 d_1'<0$ by
Lemma \ref{lem:dd'}.
Hence we can apply Lemma \ref{lem:uu} to get
$a_1' d_1-a_1 d_1'  \geq 0$.

(3) 
By (1) and (2), we get 
\begin{equation}\label{eq:area2}
\ell (a_1'd_2'-a_2' d_1')-\ell (a_1 d_2-a_2 d_1)=\ell_1 \ell_1' (a_1' d_1-a_1 d_1')+
\ell_2 \ell_2' (d_2' a_2-a_2' d_2)>0.
\end{equation}
\end{proof}



\end{document}